\documentclass[12pt,reqno]{amsart}
\usepackage{mathrsfs,amssymb,graphicx,verbatim}
\usepackage{paralist}
\usepackage[breaklinks,pdfstartview=FitH]{hyperref}
\newcommand{\U}{\mathscr{U}}
\newcommand{\G}{\mathcal{G}}

\renewcommand{\P}{{\mathscr P}}
\newcommand{\e}{\varepsilon}
\renewcommand{\d}{\delta}
\newcommand{\R}{\mathbb R}

\newcommand{\eqdef}{\stackrel{\mathrm{def}}{=}}
\newcommand{\conv}{\mathrm{conv}}
\renewcommand{\le}{\leqslant}
\renewcommand{\ge}{\geqslant}
\renewcommand{\leq}{\leqslant}

\newcommand{\B}{\mathfrak{B}}
\newtheorem{theorem}{Theorem}[section]

\newtheorem{lemma}[theorem]{Lemma}
\newtheorem{corollary}[theorem]{Corollary}

\theoremstyle{remark}

\newtheorem{remark}[theorem]{Remark}
\newtheorem{question}[theorem]{Question}
\theoremstyle{definition}
\newtheorem{definition}[theorem]{Definition}

\newcommand{\A}{\mathscr A}

\DeclareMathOperator*{\E}{\mathbb{ E}}

\DeclareMathOperator{\greenmat}{\mathcal{G}}

\newcommand{\F}{\mathcal F}
\newcommand{\N}{\mathbb N}

\renewcommand{\setminus}{\smallsetminus}

\newcommand{\Lip}{\mathrm{Lip}}

\DeclareMathOperator{\diam}{diam}

\begin{document}

\title[Spectral calculus and Lipschitz extension]{Spectral calculus and Lipschitz extension for barycentric metric spaces}

\author{Manor Mendel}
\address {Mathematics and Computer Science Department,
The Open University of Israel, 1 University Road, P.O. Box 808, Raanana 43107,Israel }
\email{manorme@openu.ac.il}.
\author{Assaf Naor}
\address{Courant Institute of Mathematical Sciences, New York University, 251 Mercer Street, New York NY 10012, USA}
\email{naor@cims.nyu.edu}


\begin{abstract}
The metric Markov cotype of barycentric metric spaces is computed, yielding the first class of metric spaces that are not Banach spaces for which this bi-Lipschitz invariant is understood. It is shown that this leads to new nonlinear spectral calculus inequalities, as well as a unified framework for Lipschitz extension, including new Lipschitz extension results for $CAT(0)$ targets.  An example  that elucidates the relation between metric Markov cotype and Rademacher cotype is analyzed, showing that a classical Lipschitz extension theorem of Johnson, Lindenstrauss and Benyamini is asymptotically sharp.
\end{abstract}



\dedicatory{Dedicated to Nigel Kalton}

\maketitle


\section{Introduction}

Our main purpose here is to compute a bi-Lipschitz invariant, called {\em metric Markov cotype}, for {\em barycentric metric spaces}; an important class of metric spaces that contains all uniformly convex Banach spaces as well as complete simply connected metric spaces that are nonpositively curved in the sense of Aleksandrov.

The notion of metric Markov cotype arises from the deep work~\cite{Ball} of K. Ball on the Lipschitz extension problem. Based mainly on Ball's ideas in~\cite{Ball}, combined with some additional geometric ingredients, we establish a fully nonlinear version of Ball's extension theorem
that allows for targets that are not necessarily Banach spaces. Due to our computation of metric Markov cotype for barycentric spaces, this yields a versatile Lipschitz extension theorem that contains as special cases many Lipschitz extension theorems that appeared in the literature, as well as Lipschitz extension results that were previously unknown.

Another use of metric Markov cotype is due to~\cite{MN-towards}, where it is shown to yield spectral calculus inequalities for nonlinear spectral gaps. Consequently, our computation of metric Markov cotype for barycentric metric spaces implies new nonlinear spectral calculus inequalities which, in the special case of $CAT(0)$ spaces, lay the groundwork for our forthcoming construction~\cite{MN12-random} of expanders with respect to certain Hadamard spaces and random graphs.

Finally, we show that a beautiful construction of
Kalton~\cite{Kal12} yields a closed linear subspace $X$ of $L_1$
(thus in particular $X$ has Rademacher cotype $2$) that fails to
have finite metric Markov cotype. By obtaining a quantitative
version of Kalton's result, we show that a classical Lipschitz
extension theorem of Johnson, Lindenstrauss and
Benyamini~\cite{JL84} is asymptotically sharp.

In order to give precise formulations of the above results one needs
to recall some background. This will be done in the subsequent
sections that contain a detailed description of the contents of this
paper.

\subsection{Markov type and metric Markov cotype}

Given $n\in \N$ and $\pi\in \Delta^{n-1}\eqdef \left\{x\in [0,1]^n:\ \sum_{i=1}^n x_i=1\right\}$, recall that a stochastic matrix $A=(a_{ij})\in M_n(\R)$ (here and in what follows, $M_n(\R)$ denotes as usual the $n$ by $n$ matrices with real entries) is said to be reversible relative to the probability vector $\pi$ if $\pi_i a_{ij}=\pi_j a_{ji}$ for all $i,j\in \{1,\ldots,n\}$. The following important definition is due to K. Ball~\cite{Ball}.
\begin{definition}[Markov type $p$]\label{def:Mtype}
A metric space $(X,d_X)$ is a said to have Markov type $p\in (0,\infty)$ with constant $M\in (0,\infty)$ if  for every $n,t\in
\N$ and every $\pi\in \Delta^{n-1}$, if $A=(a_{ij})\in M_n(\R)$ is  a stochastic matrix that is reversible relative to $\pi$ then every $x_1,\ldots,x_n\in X$ satisfy
\begin{equation}\label{eq:to reverse type}
\sum_{i=1}^n\sum_{j=1}^n \pi_i(A^t)_{ij}d_X(x_i,x_j)^p\le M^pt\sum_{i=1}^n\sum_{j=1}^n \pi_ia_{ij}d_X(x_i,x_j)^p.
\end{equation}
The infimum over those $M\in (0,\infty)$ satisfying~\eqref{eq:to reverse type} is denoted $M_p(X)$.
\end{definition}
The triangle inequality implies that $M_1(X)=1$ for every metric space $(X,d_X)$, and Ball proved in~\cite{Ball} that $M_p(\ell_p)=1$ for $p\in [1,2]$. In~\cite{NPSS06} it is shown that $M_2(\ell_p)\lesssim \sqrt{p}$ for $p\in [2,\infty)$ (here, and in what follows, $A \lesssim B$ and $B \gtrsim A$ denotes the estimate $A \leq CB$ for some absolute constant $C\in (0,\infty)$). Additional examples of computations of Markov type will be discussed in Section~\ref{sec:ext intro}.

Markov type is a bi-Lipschitz invariant that has proved itself useful to a variety problems in metric geometry, one of which will be recalled below. We refer to~\cite{Ball} for the natural probabilistic interpretation of~\eqref{eq:to reverse type} that explains the above terminology (this interpretation is not needed in the present paper, but it is important elsewhere).

\begin{definition}[metric Markov cotype $p$]\label{def:Mcotype} A metric space $(X,d_X)$ is said to have metric Markov cotype $p\in (0,\infty)$ with constant $N\in (0,\infty)$ if  for every $n,t\in
\N$ and every $\pi\in \Delta^{n-1}$, if $A=(a_{ij})\in M_n(\R)$ is  a stochastic matrix that is reversible relative to $\pi$ then for every $x_1,\ldots,x_n\in X$ there exist $y_1,\ldots,y_n\in X$
satisfying
\begin{multline}\label{def:markov cotype}
\sum_{i=1}^n \pi_id_X(x_i,y_i)^p+t\sum_{i=1}^n\sum_{j=1}^n \pi_ia_{ij}
d_X(y_i,y_j)^p\\\le N^p\sum_{i=1}^n\sum_{j=1}^n
\pi_i\left(\frac{1}{t}\sum_{s=1}^{t}A^s\right)_{ij}d_X(x_i,x_j)^p.
\end{multline}
The infimum over those $N\in (0,\infty)$ satisfying~\eqref{eq:to
reverse type} is denoted $N_p(X)$.
\end{definition}

Definition~\ref{def:Mcotype} is taken from~\cite{MN-towards}.
In~\cite{Ball} Ball suggested a seemingly different notion of Markov
cotype, but it is in fact equivalent to
Definition~\ref{def:Mcotype}, as explained in
Section~\ref{sec:ball's cotype}. Due to applications
of~\eqref{def:markov cotype} that will be described later, we
believe that it is beneficial to work with the above definition of
metric Markov cotype rather than Ball's original formulation. See
Section~\ref{sec:ball's cotype} for a description of Ball's
approach.

Condition~\eqref{def:markov cotype} originates from an attempt to introduce an invariant that is ``dual" to Markov type by reversing the inequality in~\eqref{eq:to reverse type}. However, no non-singleton metric space can satisfy~\eqref{eq:to reverse type} with the direction of the inequality reversed (this follows formally from observations in~\cite{NS02} and~\cite{MN-cotype}, and can be also easily verified directly). \eqref{def:markov cotype} achieves a similar reversal of~\eqref{eq:to reverse type} by allowing one to pass from the initial points
$x_1,\ldots,x_n\in X$ to new points $y_1,\ldots,y_n\in X$. The
first summand in the left hand side of~\eqref{def:markov cotype}
ensures that on average (with respect to $\pi$) $y_i$ is close to $x_i$. The remaining terms
in~\eqref{def:markov cotype} correspond to the reversal
of~\eqref{eq:to reverse type}, with $\{x_i\}_{i=1}^n$ replaced by
$\{y_i\}_{i=1}^n$  in the left hand side, and the power $A^t$
replaced by the Ces\`aro average
$\frac{1}{t}\sum_{s=1}^{t}
A^{s}$.

Due to~\cite{Ball,MN-towards}, Banach spaces that admit an equivalent norm whose modulus of convexity has power type $p$ have metric Markov cotype $p$, in particular $N_p(\ell_p)\lesssim 1$ for $p\in [2,\infty)$ and $N_2(\ell_p)\lesssim 1/\sqrt{p-1}$ for $p\in (1,2]$. Prior to the present work this was the only nontrivial class of metric spaces whose metric Markov cotype was known. Here we enrich the repertoire of metric spaces for which one can prove a metric Markov cotype inequality such as~\eqref{def:markov cotype}, treating also spaces that are not necessarily Banach spaces.

\subsection{Barycentric metric spaces}
In order to avoid measurability considerations that are irrelevant to the discussion at hand, we will tacitly assume throughout this article that all measures are finitely supported and all $\sigma$-algebras are finite.

The set of probability measures on a set $X$ is denoted $\mathscr{P}_X$. Denoting the point mass at $x\in X$ by $\delta_x\in \P_X$, every $\mu\in \P_X$ can be written uniquely as $\mu=\sum_{i=1}^n\lambda_i\d_{x_i}$ for some $n\in \N$, distinct $x_1,\ldots,x_n\in X$ and $(\lambda_1,\ldots,\lambda_n)\in \Delta^{n-1}\cap (0,1]^n$. A coupling of $\mu,\nu\in \P_X$ is a measure $\pi\in \P_{X\times X}$ such that $\sum_{z\in X}\pi(x,z)=\mu(x)$ and $\sum_{z\in X}\pi(z,y)=\nu(y)$ for every $x,y\in X$ (both of these sums are finite). The set of all the couplings of $\mu$ and $\nu$ is denoted $\Pi(\mu,\nu)\subseteq \P_{X\times X}$. If $(X,d_X)$ is a metric space and $p\in [1,\infty)$ then the corresponding Wasserstein $p$ metric on $\P_X$ is defined as usual by
$$
\forall\, \mu,\nu\in \P_X,\qquad W_p(\mu,\nu)\eqdef \inf_{\pi\in \Pi(\mu,\nu)} \left(\int_{X\times X} d_X(x,y)^pd\pi(x,y)\right)^{1/p}.
$$
\begin{definition}[$W_p$ barycentric metric space]\label{def:Wp}
Fix $p,\Gamma\in [1,\infty)$. A metric space $(X,d_X)$ is said to be $W_p$ barycentric with constant $\Gamma$ if there exists a mapping $\B:\P_X\to X$ satisfying $\B(\d_x)=x$ for every $x\in X$, and
\begin{equation}\label{eq:Gamma Wp Lip}
\forall\, \mu,\nu\in \P_X,\qquad d_X(\B(\mu),\B(\nu))\le \Gamma W_p(\mu,\nu).
\end{equation}
$(X,d_X)$ is said to be $W_p$ barycentric if it is $W_p$ barycentric with constant $\Gamma$ for some $\Gamma\in [1,\infty)$.
\end{definition}

In what follows, a mapping $\B:\P_X\to X$ satisfying $\B(\d_x)=x$ for every $x\in X$ will be called a {\em barycenter map}.
The notion of $W_p$ barycentric metric spaces was studied  by several authors: see e.g. \cite{LPS-cat0-ext,EH99,Stu03,Gro03,Oht09,Nav12}. Note that if $(X,d_X)$ is $W_p$ barycentric with constant $\Gamma$ then it is also $W_q$ barycentric with constant $\Gamma$ for every $q\ge p$. Normed spaces are $W_1$ barycentric with constant $\Gamma=1$, as exhibited by the barycenter map $\B(\mu)=\int_X xd\mu(x)$. Metric spaces that are nonpositively curved in the sense of Busemann (see~\cite{BH-book}) are also $W_1$ barycentric with constant $\Gamma=1$, as shown in~\cite{EH99,Nav12}.

\begin{definition}[$p$-barycentric metric space]\label{def:p}  Fix $p,K\in [1,\infty)$. A metric space $(X,d_X)$ is said to be $p$-barycentric with constant $K$ if there exists a
 mapping $\B:\P_X\to X$ such that for every $x\in X$ and $\mu\in \P_X$ we have
\begin{equation}\label{eq:def p bary}
d_X\left(x,\B(\mu)\right)^p+\frac{1}{K^p}\int_X d_X(\B(\mu),y)^pd\mu(y)\le \int_X d_X(x,y)^pd\mu(y).
\end{equation}
$(X,d_X)$ is said to be $p$-barycentric if it is $p$-barycentric with constant $K$ for some $K\in [1,\infty)$.
\end{definition}
The appearance of the constant $K$ in the left hand side of~\eqref{eq:def p bary} is natural from the point of view of Banach space theory (see~\cite[Lem~3.1]{Ball} and~\cite[Lem.~6.5]{MN-towards}), but note that it means that (unless $K=1$) $\B(\mu)$ need not be a point $x\in X$ which minimizes the right hand side of~\eqref{eq:def p bary}. In many cases of interest one defines barycenters as minimizers of the right hand side of~\eqref{eq:def p bary}, but as will become clear from the ensuing considerations, Definition~\ref{def:p} suffices for many purposes.

In what follows, whenever we say that a metric space is $W_p$ barycentric with constant $\Gamma$ and also $q$-barycentric with constant $K$ we mean that Definition~\ref{def:Wp} and Definition~\ref{def:p} are satisfied with respect to {\em the same} barycenter map $\B:\P_X\to X$. Examples of such spaces include all complete $CAT(0)$ metric spaces (Hadamard spaces), which are $W_1$ barycentric with constant $1$ (see~\cite[Lem.~4.2]{LPS-cat0-ext} or~\cite[Thm.~6.3]{Stu03})  and also $2$-barycentric with constant $1$ (see~\cite[Lem.~4.1]{LPS-cat0-ext} or~\cite[Lem.~4.4]{Stu03}). Banach spaces whose modulus of uniform convexity have power type $p\in [2,\infty)$ are $W_1$ barycentric with constant $1$ and also $p$-barycentric (see~\cite[Lem~3.1]{Ball} for $p=2$ and~\cite[Lem.~6.5]{MN-towards} for $p\in [2,\infty)$).

We refer to the books~\cite{Bal95,Jos97,BH-book} for an extensive
discussion of the important class of $CAT(0)$ metric spaces, which
includes e.g. complete simply connected Riemannian manifolds with
nonpositive sectional curvature and Euclidean Tits buildings. For
the sake of readers who are not familiar with this notion we state
that the definition of the class of $CAT(0)$ metric spaces can be
taken to be those metric spaces $(X,d_X)$ for which there exists a
mapping $\B:\P_X\to X$ that satisfies~\eqref{eq:def p bary} with
$p=2$ and $K=1$ for probability measures $\mu$ that are supported on
at most two points~\cite[Thm.~4.9]{Stu03}. Readers who are not
familiar with the theory of uniformly convex Banach spaces are
referred to~\cite{Fiegel76,BCL}.


\subsection{Metric Markov cotype for barycentric metric spaces}

In Section~\ref{sec:cotype proof} we prove the following result.

\begin{theorem}\label{thm:bary-cotype} Fix $p,K,\Gamma\in [1,\infty)$. Suppose that $(X,d_X)$ is a metric space that is $W_p$ barycentric with constant $\Gamma$ and also $p$-barycentric with constant $K$. Then $(X,d_X)$ has metric Markov cotype $p$ with
\begin{equation}\label{eq:NpGammaK}
N_p(X)\lesssim \Gamma K.
\end{equation}
\end{theorem}

The special case of Theorem~\ref{thm:bary-cotype} when $X$ is a Banach space whose modulus of uniform convexity has power type $p$ was proved in~\cite{Ball,MN-towards}. Our proof of Theorem~\ref{thm:bary-cotype} is based on an extension of the method of~\cite{MN-towards} to the present nonlinear setting. In particular  we prove for this purpose a nonlinear analogue of Pisier's martingale cotype inequality~\cite{Pisier-martingales}; see Section~\ref{sec:Pisier} below.

\begin{remark}
The property of having metric Markov cotype $p$ is clearly a
bi-Lipschitz invariant. Similarly, the property of being $W_p$
barycentric is a bi-Lipschitz invariant, but this is not the case
for the property of being $p$-barycentic. Thus
Theorem~\ref{thm:bary-cotype} leaves something to be desired, since
its assumption is not invariant under bi-Lipschitz deformations
while its conclusion is. By examining the proof of
Theorem~\ref{thm:bary-cotype} one can extract a somewhat tedious
bi-Lipschitz invariant condition that implies the same
conclusion~\eqref{eq:NpGammaK}. It would be interesting to obtain a
clean intrinsic characterization of those metric spaces $(X,d_X)$
that are bi-Lipschitz equivalent to a $p$-barycentric metric space.
For Banach spaces this was done in~\cite{MN13}, the desired metric
invariant being the notion of {\em Markov $p$-convexity}
(see~\cite{MN13} for the definition). The method of~\cite{MN13}
relies on the Banach space structure, so it remains open to
characterize intrinsically those $W_p$ barycentric metric spaces
that are bi-Lipschitz equivalent to a $p$-barycentric metric space. It would also be interesting to characterize those Finsler manifolds that are $p$-barycentric.
\end{remark}

\subsection{Calculus for nonlinear spectral gaps}\label{sec:calculus}

Let $A=(a_{ij})\in M_n(\R)$ be a symmetric stochastic matrix. Denote the decreasing rearangment of the eigenvalues of $A$  by $1=\lambda_1(A)\ge \lambda_2(A)\ge \lambda_n(A)\ge -1$, and write $\lambda(A)=\max_{i\in \{2,\ldots,n\}} |\lambda_i(A)|$.

Following~\cite{MN-towards}, given a metric space $(X,d_X)$ and $p\in (0,\infty)$ let $\gamma_+(A,d_X^p)$ denote the infimum over those $\gamma_+\in (0,\infty]$ for which every $x_1,\ldots,x_n,y_1,\ldots,y_n\in X$ satisfy
\begin{equation}\label{eq:def gamma+}
\frac{1}{n^2}\sum_{i=1}^n\sum_{j=1}^n d_X(x_i,y_j)^p\le \frac{\gamma_+}{n}\sum_{i=1}^n\sum_{j=1}^n a_{ij} d_X(x_i,y_j)^p.
\end{equation}

Letting $d_\R$ denote the standard metric on $\R$, i.e.,
$d_\R(x,y)=|x-y|$, by simple linear algebra we see that
$\gamma_+(A,d_\R^2)=1/(1-\lambda(A))$. One should therefore think of
the quantity $\gamma_+(A,d_X^p)$ as measuring the magnitude of the
 {\em nonlinear absolute spectral gap} of the matrix $A$ with respect
to the geometry of $X$. We refer to~\cite{MN-towards} for a detailed
discussion of nonlinear spectral gaps and their applications.

Despite the fact that we call inequalities such as~\eqref{eq:def gamma+} ``spectral inequalities", there is no actual spectrum present here, and therefore tools that are straightforward in the linear setting due to the link to linear algebra fail to hold true in general. This is especially important in the context of {\em nonlinear spectral calculus}, where one aims to relate $\gamma_+(A^t,d_X^p)$ to $\gamma_+(A,d_X^p)$. We refer to~\cite{MN-towards} for an explanation of the importance of this problem, where the following theorem is proved.

\begin{theorem}\label{thm:quote MN} There exists a universal constant $\kappa\in (0,\infty)$ with the following property. Suppose that $p\in [1,\infty)$ and that $(X,d_X)$ is a metric space that has metric Markov cotype $p$. Then for every $n,t\in \N$, every symmetric stochastic matrix $A\in M_n(\R)$ satisfies
$$
\gamma_+\left(\frac{1}{t}\sum_{s=1}^t A^s,d_X^p\right)\le (\kappa N_p(X))^p\max\left\{1,\frac{\gamma_+(A,d_X^p)}{t}\right\}.
$$
\end{theorem}

By combining Theorem~\ref{thm:bary-cotype} and Theorem~\ref{thm:quote MN} we conclude that the following result holds true.

\begin{theorem}\label{thm:calculus cesaro}
There exists a universal constant $c\in (1,\infty)$ such that for every $p,K,\Gamma\in [1,\infty)$, if $(X,d_X)$ is a metric space that is $W_p$ barycentric with constant $\Gamma$ and $p$-barycentric with constant $K$ then for every $n,t\in \N$, every symmetric stochastic matrix $A\in M_n(\R)$ satisfies
\begin{equation}\label{eq:bary calculus ineq}
\gamma_+\left(\frac{1}{t}\sum_{s=1}^t A^s,d_X^p\right)\le (c\Gamma K)^p \max\left\{1,\frac{\gamma_+(A,d_X^p)}{t}\right\}.
\end{equation}
\end{theorem}

For future applications it is worthwhile to single out the following special case of Theorem~\ref{thm:calculus cesaro}.

\begin{corollary}\label{coro:cat0 calculus}
Suppose that $(X,d_X)$ is a $CAT(0)$ space. Then
for every $n,t\in \N$, every symmetric stochastic matrix $A\in M_n(\R)$ satisfies
\begin{equation*}\label{eq:calculus cat0}
\gamma_+\left(\frac{1}{t}\sum_{s=1}^t A^s,d_X^2\right)\lesssim \max\left\{1,\frac{\gamma_+(A,d_X^2)}{t}\right\}.
\end{equation*}
\end{corollary}

Corollary~\ref{coro:cat0 calculus} was the main motivation for the investigations that led to the present paper, since it plays a key role in our forthcoming work~\cite{MN12-random} that establishes for the first time the existence of expanders with respect to certain Hadamard spaces and random graphs.

For the purpose of the applications in~\cite{MN12-random}, the fact that the spectral calculus inequality~\eqref{eq:bary calculus ineq} involves  Ces\`aro averages of $A$ rather than powers of $A$ is immaterial, but it is natural to ask if it is possible to relate $\gamma_+(A^t,d_X^p)$ to $\gamma_+(A,d_X^p)$. In the setting of general barycentric metric spaces  this question remains open, but for $CAT(0)$ spaces, or more generally under the requirement $K=1$ in~\eqref{eq:def p bary}, it is indeed possible to do so, albeit via an upper bound on $\gamma_+(A^t,d_X^p)$ in terms of $\gamma_+(A,d_X^p)$ that is weaker than the right hand side of~\eqref{eq:bary calculus ineq}.

\begin{theorem}\label{thm:calculus gap} There is a universal constant $C\in (0,\infty)$ with the following property. Fix $p,\Gamma\in [1,\infty)$ and suppose that $(X,d_X)$ is a metric space that is $W_p$ barycentric with constant $\Gamma$ and $p$-barycentric with constant $K=1$. Then for every $n,t\in \N$, every symmetric stochastic matrix $A\in M_n(\R)$ satisfies
$$
\gamma_+\left(A^t,d_X^p\right)\le (C\Gamma )^{p}\left(\max\left\{1,p\cdot \frac{ \gamma_+(A,d_X^p)}{t}\right\}\right)^p.
$$
\end{theorem}
Our proof of Theorem~\ref{thm:calculus gap} relies on ideas from~\cite[Sec.~6]{MN-towards}, where a similar treatment is given to uniformly convex Banach spaces (in this special context the conclusion of Theorem~\ref{thm:calculus gap} holds true even without the restriction $K=1$). In the present nonlinear setting several modifications of the argument of~\cite{MN-towards} are required; see Section~\ref{sec: norm} below.

\subsection{Lipschitz extension}\label{sec:ext intro}

If $(X,d_X)$ and $(Y,d_Y)$ are metric spaces then for every
$S\subseteq X$ denote by $e(X,S,Y)$ the infimum over those $L\in
(0,\infty)$ such that for {\em every} Lipschitz function $f:S\to Y$
there exists $F:X\to Y$ with $F(x)=f(x)$ for every $x\in S$ such that $\|F\|_{\Lip}\le L\|f\|_{\Lip}$,
where $\|f\|_{\Lip}$ denotes the Lipschitz constant of $f$. If no
such $L$ exists then set $e(X,S,Y)=\infty$.

Defining
$$
e(X,Y)\eqdef \sup\left\{e(X,S,Y):\ S\subseteq X\right\},
$$
the goal of the {\em Lipschitz extension problem} is to understand
which pairs of metric spaces $(X,d_X), (Y,d_Y)$ satisfy
$e(X,Y)<\infty$, and, when that happens, to obtain good bounds on
$e(X,Y)$. Due to its intrinsic importance as well as many
applications in analysis and geometry, the Lipschitz extension
problem has been extensively investigated over the past century. We
shall not attempt to indicate the vast literature on this topic,
referring instead to the book~\cite{BB12} and the references
therein.

K. Ball introduced~\cite{Ball}   the notions of Markov type and
cotype in order to prove an important Lipschitz extension theorem
known today as {\em Ball's extension theorem}. Based  on Ball's
ideas in~\cite{Ball}, the following result is proved in
Section~\ref{sec:ext}.

\begin{theorem}[generalized Ball extension theorem]\label{thm:general keith}
Fix $p,\Gamma\in [1,\infty)$. Let $(X,d_X)$ be a metric space of Markov type $p$ and let $(Y,d_Y)$ be a metric space of metric Markov cotype $p$ that is $W_p$ barycentric with constant $\Gamma$. Suppose that $Z\subseteq X$ and $f:Z\to Y$ is Lipschitz. Then for every finite subset $S\subseteq X$ there exists $F:S\to Y$ with $F|_{S\cap Z}=f|_{S\cap Z}$ and
\begin{equation}\label{eq:extension lip bound}
\|F\|_{\Lip} \lesssim \Gamma M_p(X)N_p(Y)\|f\|_{\Lip}.
\end{equation}
\end{theorem}

By combining Theorem~\ref{thm:general keith} with Theorem~\ref{thm:bary-cotype} we deduce the following Lipschitz extension result.

\begin{corollary}\label{eq:extension bary}
Fix $p,K,\Gamma\in [1,\infty)$. Suppose that $(X,d_X)$ is a metric space of Markov type $p$ and that $(Y,d_Y)$ is a metric space that is $W_p$ barycentric with constant $\Gamma$ and also $p$-barycentric with constant $K$. Suppose that $Z\subseteq X$ and $f:Z\to Y$ is Lipschitz. Then for every finite subset $S\subseteq X$ there exists $F:S\to Y$ with $F|_{S\cap Z}=f|_{S\cap Z}$ and
\begin{equation*}
\|F\|_{\Lip} \lesssim \Gamma^2K M_p(X)\|f\|_{\Lip}.
\end{equation*}
\end{corollary}

In~\cite{Ball} Ball obtained the conclusion of Theorem~\ref{thm:general keith} when $Y$ is a Banach
space, under the assumption that it satisfies a certain {\em linear}
invariant that he called Markov cotype $2$. He also proved that
Banach spaces that admit an equivalent norm whose modulus of uniform
convexity has power type $2$ satisfy this assumption.

In~\cite[Sec.~6]{Ball}, Ball proposed a way to define Markov cotype
$2$ for metric spaces: he first defined a bi-Lipschitz invariant of
metric spaces that he called ``approximate convexity", and for
approximately convex metric spaces he defined a notion of metric Markov
cotype which is the same as~\eqref{def:markov cotype}, except that
in the right hand side of~\eqref{def:markov cotype} the Ces\`aro
average of $A$ is replaced by a certain Green's matrix corresponding
to $A$. The precise formulation of these concepts is recalled in
Section~\ref{sec:ball's cotype}, where we show that Ball's notion of
metric Markov cotype coincides with the notion of metric Markov cotype as
in Definition~\ref{def:markov cotype}. Our contribution here is to
show that Ball's strategy yields the desired Lipschitz extension
result, with the following differences: the $W_p$ barycentric
condition is used in a key duality step (Lemma~\ref{lem:ball 1.1}
below), and Lemma~\ref{lem:CN} below removes the need to use the
notion of approximate convexity. Other than these changes and
some expository simplifications, Section~\ref{sec:ext} is nothing
more than a realization of Ball's plan as he
 originally envisaged it.

Theorem~\ref{thm:general keith} yields an extension of $f$ to
finitely many additional points, with a bound on the Lipschitz
constant that is independent of the number of the additional points. This result is the
main geometric content of the Lipschitz extension phenomenon studied
here, but using standard arguments one can formally deduce from
Theorem~\ref{thm:general keith} bona fide solutions of the Lipschitz
extension problem.

Specifically, let $I$ denote the set of all finite subsets of $X$
and let $\U$ be a free ultrafilter on $I$. Denoting by $Y_\U$ the
associated ultrapower of $Y$ (see~\cite{KL95} for background on
ultrapowers of metric spaces), $Y$ is canonically embedded in $Y_\U$
and it follows formally from Theorem~\ref{thm:general keith} that
there exists a mapping $\Phi:X\to Y_\U$ that extends $f$ and
satisfies~\eqref{eq:extension lip bound}. If for some $\lambda\in
[1,\infty)$ there were a $\lambda$-Lipschitz retraction from $Y_\U$
onto $Y$, then by composing $\Phi$ with this retraction we would deduce
that
\begin{equation}\label{eq:lambda retract}
e(X,Y)\lesssim \lambda\Gamma M_p(X)N_p(Y).
\end{equation}

If $Y$ is Banach space then $Y_\U$ is also a Banach space, and, as
proved in~\cite{Hei80}, it follows from the principle of local
reflexivity~\cite{LR69,JRZ71} that there is a linear isometry
$T:Y^{**} \to Y_\U$ such that $T(Y^{**})$ contains the canonical
image of $Y$ in $Y_\U$, and there is a norm $1$ projection of $Y_\U$
onto $T(Y^{**})$. It therefore follows from Theorem~\ref{thm:general
keith} that $$e(X,Y^{**})\lesssim \Gamma M_p(X)N_p(Y).$$ If in
addition there is a $\lambda$-Lipschitz retraction from $Y^{**}$
onto $Y$ then it would follow that~\eqref{eq:lambda retract} holds
true.

It is a long standing open problem whether for every separable
Banach space $Y$ there is a Lipschitz retraction from $Y^{**}$ onto
$Y$, but in the nonseparable setting it has been recently proved by
Kalton~\cite{Kal11} that this need not hold true. A dual Banach
space is always canonically norm $1$ complemented in its bi-dual,
and in~\cite[Sec.~5]{Kal12} Kalton proved that if $Y$ either has an
unconditional finite dimensional decomposition (UFDD) or is a
separable order continuous Banach lattice then there is a Lipschitz
retraction from $Y^{**}$ onto $Y$.

If $Y$ is a complete $CAT(0)$ metric space then so is $Y_\U$, and
moreover $Y$ is a closed convex subset of $Y_\U$. In this case there
is a $1$-Lipschitz retraction from $Y_\U$ onto $Y$ (the nearest
point map); see~\cite[Ch.~II.2]{BH-book}.

The above discussion yields a variety of target spaces $Y$ for which
the assumptions of Theorem~\ref{thm:general keith} implies that
$e(X,Y)<\infty$. We single out in particular the following statement.

\begin{corollary}\label{coro:bona fide lip ext}
Under the assumptions of Theorem~\ref{thm:general keith}, if in
addition $Y$ is a dual Banach space  then $e(X,Y)\lesssim
M_p(X)N_p(Y)$. If $p=2$ and $Y$ is a complete $CAT(0)$ metric space then
$e(X,Y)\lesssim M_2(X)$.
\end{corollary}

The Markov type of several important classes of metric spaces
has been computed in the literature, and when one takes $(X,d_X)$ to
be one of those spaces Corollary~\ref{coro:bona fide lip ext}
becomes a versatile Lipschitz extension theorem that encompasses a
wide range of seemingly disparate Lipschitz extension results, that
have been previously proved mostly via completely different methods.

Specifically, in~\cite{NPSS06} it was proved that Banach spaces that
admit an equivalent norm whose modulus of uniform smoothness has
power type $p$ have Markov type $p$. It was also proved
in~\cite{NPSS06} that trees, hyperbolic groups, complete simply
connected Riemannian manifolds of pinched sectional curvature and
Laakso graphs all have Markov type $2$, and that spaces that admit a
padded random partition (see~\cite{LN-extension}), in particular
doubling metric spaces and planar graphs, have Markov type $p$ for
all $p\in (0,2)$. In~\cite{BKL07} it was shown that series parallel
graphs have Markov type $2$, and finally in the recent
work~\cite{DLP12} it was shown that spaces that admit a padded
random partition have Markov type $2$. Thus, in particular, doubling
spaces and planar graphs have Markov type $2$. In~\cite{NS11} it was
shown that spaces with finite Nagata dimension admit a padded random
partition, and so by~\cite{DLP12} they too have Markov type $2$.
In~\cite{Oht09-markov} it was shown that Aleksandrov spaces of
nonnegative curvature have Markov type $2$, and in~\cite{ANN10} the
Markov type of certain Wasserstein spaces was computed.

In light of these results, taking as an example the case when
$(Y,d_Y)$ is a Hadamard space in Corollary~\ref{coro:bona fide lip
ext}, we see that if $(X,d_X)$ is a doubling space, planar graph, or
a space with finite Nagata dimension, then $e(X,Y)$ is finite. These
results were previously proved in~\cite{LN-extension} via the method
of random partitions (Lipschitz extension for spaces of bounded
Nagata dimension was previously treated in~\cite{LS05} and only
later it was shown in~\cite{NS11} that they admit a padded random
partition and therefore the corresponding extension results are a
special case of~\cite{LN-extension}). It also follows that if
$(X,d_X)$ has nonnegative curvature in the sense of Aleksandrov and
$(Y,d_Y)$ is a Hadamard space then $e(X,Y)\lesssim 1$, a result that
has been previously proved in~\cite{LS97}, as a special case of an
elegant generalization of the classical Kirszbraun extension
theorem~\cite{Kirsz34}.

Given a metric space $(X,d_X)$ and $\alpha\in (0,1]$ let $X^\alpha$
denote the metric space $(X,d_X^\alpha)$. By the triangle inequality
$X^\alpha$ has Markov type $p$ with constant $1$ for every $p\in
(1,1/\alpha]$. It therefore follows from  the above discussion that
$e(X^\alpha,Y)<\infty$ for every metric space $X$, provided that
 $Y$ has metric Markov cotype $p\in (1,1/\alpha]$ and
there is a Lipschitz retraction from $Y_{\mathscr{U}}$ onto $Y$. In
particular, every $1/2$-H\"older mapping from a subset of a metric
space $X$ into a Hadamard space $Y$ can be extended to a $Y$-valued
$1/2$-H\"older mapping defined on all of $X$; this statement was
previously known when $Y$ is a Hilbert space due to the work of
Minty~\cite{Min70}. One can state several additional examples of
this type, but we single out only one more special case of
Corollary~\ref{coro:bona fide lip ext} that does not seem to follow
from previously known theorems: if $X$ is a Banach space whose
modulus of smoothness has power type $2$ (thus by~\cite{NPSS06} $X$ has Markov type $2$), e.g. $X$ can be an
$L_p(\mu)$ space or the Schatten trace class $S_p$ for $p\in
[2,\infty)$, and $Y$ is a Hadamard space, then $e(X,Y)<\infty$.

\subsection{On a construction of Kalton}

Kalton recently used his ``method of sections" to obtain several
striking results on the nonlinear geometry of Banach spaces. Using
Kalton's beautiful work in~\cite{Kal12}, we prove the following
result in Section~\ref{sec:kalton}.

\begin{theorem}\label{thm:no cotype}
There exists a closed linear subspace of $\ell_1$ that fails to
have metric Markov cotype $p$ for every $p\in (0,\infty)$.
\end{theorem}

Much of the impetus for research on bi-Lipschitz invariants stems from the search for nonlinear formulations of key concepts in Banach space theory; see the surveys~\cite{Bal12,Nao12} and the references therein for more on this program. In particular, the use of the term ``cotype" in Definition~\ref{def:Mcotype} arises from an analogy with the Banach space notion of {\em Rademacher cotype} (see e.g.~\cite{Mau03}). $\ell_1$, and hence all of its linear subspaces, has Rademacher cotype $2$, so Theorem~\ref{thm:no cotype} shows that for Banach spaces metric Markov cotype and Rademacher cotype are different notions. Nevertheless, it would be very interesting to understand the metric Markov cotype of $\ell_1$ itself rather than its closed subspaces (note that, due to the existential quantifier in Definition~\ref{def:Mcotype}, metric Markov cotype is not trivially inherited by subspaces).

\begin{question}\label{Q:does l1 have cotype}
Does $\ell_1$ have metric Markov cotype $2$? Less ambitiously, does $\ell_1$ have metric Markov cotype $p$ for some $p\in [2,\infty)$?
\end{question}

If $\ell_1$ had metric Markov cotype $2$ then it would follow from
Corollary~\ref{coro:bona fide lip ext} that
$e(\ell_2,\ell_1)<\infty$. Whether or not $e(\ell_2,\ell_1)$ is
finite is a long-standing open question that was asked by Ball
in~\cite{Ball}; see~\cite{MM10} for algorithmic ramifications of
this important question.

The proof of Theorem~\ref{thm:no cotype} yields the following quantitative statement. For every $n\in \N$ there exists an $n$-dimensional subspace $Z_n$ of $\ell_1$ such that
\begin{equation}\label{eq:N2 Z_n lower}
N_2(Z_n)\gtrsim \sqrt[4]{\log n}.
\end{equation}
Any $n$-dimensional subspace $X$ of $\ell_1$ satisfies
$N_2(X)\lesssim \sqrt{\log n}$. Indeed, by~\cite{Tal90} we know that
$X$ is $2$-isomorphic to a subspace of $\ell_1^k$, with $k\lesssim
n\log n$ (for our purpose  we can also use the weaker bound on $k$
of~\cite{Sch87}). By H\"older's inequality $\ell_1^k$ is
$O(1)$-isomorphic to a subspace of $\ell_p$ with $p=1+1/\log k$, so
the desired upper bound on $N_2(X)$ follows
from~\cite{Ball,MN-towards}. We ask whether~\eqref{eq:N2 Z_n lower}
can be sharpened.

\begin{question}
Is it true that for arbitrarily large $n\in \N$ there exists an $n$-dimensional subspace $X$ of $\ell_1$ with $N_2(X)\gtrsim \sqrt{\log n}$?
\end{question}

An interesting byproduct of our quantitative analysis of Kalton's
construction is that it shows for the first time that an old
Lipschitz extension result of Benyamini, Johnson and
Lindenstrauss~\cite{JL84} cannot be improved. Given $\e\in (0,1)$
and spaces $(X,d_X),(Y,d_Y)$, denote
\begin{multline*}
e_\e(X,Y)\\\eqdef \sup\left\{e(X,S,Y):\ S\subseteq X,\ \mathrm{and}\ \inf_{\substack{x,y\in S\\x\neq y}}d_X(x,y)\ge \e\diam(S)\right\},
\end{multline*}
where $\diam(S)=\sup_{x,y\in S} d_X(x,y)$ is the diameter of $S$. In
other words, we are interested in the extension of $Y$-valued
Lipschitz functions from $\e$-separated subsets of $X$, where
$\e$-separated means that all positive distances in the subset are
at least an $\e$-fraction of its diameter.

In~\cite{JL84} it was shown that for every $\e\in (0,1)$, every metric space $(X,d_X)$ and every Banach space $(Y,d_Y)$ we have
\begin{equation}\label{eq:JLB}
e_\e(X,Y)\lesssim \frac{1}{\e}.
\end{equation}
Specifically, a first proof of~\eqref{eq:JLB} was given by Johnson and Lindenstrauss in~\cite{JL84} when $Y$ is a Hilbert space, and in the appendix of the same paper Johnson and Lindenstrauss include a different argument that was subsequently found by Benyamini establishing~\eqref{eq:JLB} when $Y$ is a general Banach space. A very short proof of~\eqref{eq:JLB} was later found by Johnson, Lindenstrauss and Schechtman~\cite{JLS86}.

Johnson and Lindenstrauss proved~\cite{JL84} that
$e_\e(\ell_1,\ell_2)\gtrsim 1/\sqrt[4]{\e}$. Constructions of
Johnson, Lindenstrauss and Schechtman~\cite{JLS86} and
Lang~\cite{Lan99} yield the estimate $e_\e
(\ell_\infty,\ell_2)\gtrsim 1/\sqrt{\e}$. Here we show
that~\eqref{eq:JLB} is sharp up to absolute constant factors, even
when $X$ is Hilbert space and $Y$ is an appropriately chosen closed
subspace of $\ell_1$.

\begin{theorem}\label{thm:aspect}
There exists a closed subspace $Y$ of $\ell_1$ that satisfies $e_\e(\ell_2,Y)\gtrsim 1/\e$ for every $\e\in (0,1)$. Specifically, for every $n\in \N$ there exists a $5^n$-dimensional subspace $Y_n$ of $\ell_1$ and a $1/\sqrt[4]{n}$ net $\mathcal{N}$ of the unit ball of $\ell_2^n$ such that $e(\ell_2^n, \mathcal{N},Y_n)\gtrsim \sqrt[4]{n}$.
\end{theorem}

It would be very interesting to understand those pairs of Banach spaces $X, Y$ for which $e_\e(X,Y)=o(1/\e)$ as $\e\to 0$. We do not even know if there exist Banach spaces $X,Y$ such that $e_\e(X,Y)=o(1/\e)$ yet $e(X,Y)=\infty$. Our interest in this natural question is partially motivated by the forthcoming work~\cite{ANN11}, where it is asked whether $e_\e(\ell_1,\ell_1)=o(1/\e)$, and it is shown that
a positive answer to this question would have applications to dimension
reduction in $\ell_1$ (e.g., it is shown in~\cite{ANN11} that
if  $e_\e(\ell_1,\ell_1)=o(1/\e)$ then any $n$-point subset of $\ell_1$ embeds with distortion $O(1)$ into some Banach space of dimension $(\log n)^{O(1)}$). Due to Theorem~\ref{thm:aspect}, one is tempted to believe that in fact $e_\e(\ell_1,\ell_1)\gtrsim 1/\e$, but the present approach does not seem to shed light on this question.

\section{Pisier's martingale inequality in barycentric spaces}\label{sec:Pisier}

Martingales  in metric spaces have been studied for several decades; see e.g.~\cite{Dos62,Eme89,EH99,Sturm02,CS08}. Here we will use a natural notion of martingale in barycentric metric spaces, the main goal being to extend an important martingale inequality of Pisier~\cite{Pisier-martingales} from the setting of uniformly convex Banach spaces to the setting of barycentric metric spaces. This inequality will be used crucially in the proof of Theorem~\ref{thm:bary-cotype}.


Let $\Omega$ be a finite set and $\mu\in \P_\Omega$ be a probability
measure such that $\mu(\omega)>0$ for every $\omega\in \Omega$.
Suppose that $(X,d_X)$ is a metric space and fix a barycenter map
$\B:\P_X\to X$. Let $\F\subseteq 2^\Omega$ be a $\sigma$-algebra. For
every $\omega\in \Omega$ let $\F(\omega)\subseteq \Omega$ be the
unique atom of $\F$ to which $\omega$ belongs. Given an $X$-valued
random variable $Z:\Omega\to X$, its conditional barycenter
$\B(Z|\F):\Omega\to X$ is defined as
\begin{equation}\label{eq:def conditional}
\forall\, \omega\in \Omega,\quad \B(Z|\F)(\omega)\eqdef \B\left(\frac{1}{\mu(\F(\omega))}\sum_{a\in \F(\omega)}\mu(a)\d_{Z(a)}\right).
\end{equation}
If $m\in \N$ and $\{\Omega,\emptyset\}=\F_0\subseteq \F_1\subseteq \ldots\subseteq \F_m\subseteq 2^\Omega$ are increasing $\sigma$-algebras (a filtration) then a sequence of $X$-valued random variables $Z_0,\ldots,Z_n:\Omega\to X$ is said to be a martingale if for every $i\in \{1,\ldots,m\}$ we have $\B(Z_i|\F_{i-1})=Z_{i-1}$. We warn that in contrast to the usual setting of martingales in Banach spaces, this definition does not necessarily imply that $\B(Z_i|\F_j)=Z_j$ for every $j\in \{0,\ldots,i-2\}$. Nevertheless, the above notion of martingale suffices to prove the following inequality.

\begin{lemma}[Pisier's inequality for barycentric spaces]\label{lem:pisier}
Fix $m\in \N$, $p,K\in [1,\infty)$ and a metric space $(X,d_X)$ that is $p$-barycentric with constant $K$. Let $\Omega$ be a finite set and fix $\mu\in \P_\Omega$ with $\mu(\omega)>0$ for every $\omega\in \Omega$. Suppose that $\{\Omega,\emptyset\}=\F_0\subseteq \F_1\subseteq \ldots\subseteq \F_m\subseteq 2^\Omega$ is a filtration with respect to which $Z_0,\ldots,Z_n:\Omega\to X$ is an $X$-valued martingale. Then for every $z\in X$ we have
\begin{equation}\label{eq:pisier bary}
K^pd_X(Z_0,z)^p+\sum_{t=0}^{m-1} \int_\Omega d_X(Z_{t+1},Z_t)^p d\mu\le K^p\int_\Omega d_X(Z_m,z)^pd\mu.
\end{equation}
\end{lemma}

\begin{proof}
Fix $t\in \{0,\ldots,m-1\}$ and $\omega\in \Omega$.
Recalling~\eqref{eq:def conditional}, an application
of~\eqref{eq:def p bary} to the probability measure
$\frac{1}{\mu(\F_t(\omega))}\sum_{a\in \F_t(\omega)}
\mu(a)\d_{Z_{t+1}(a)}$ yields the estimate
\begin{multline}\label{eq:use p bary conditioned}
d_X(Z_{t}(\omega),z)^p+\frac{1}{K^p} \sum_{a\in \F_t(\omega)} \frac{\mu(a)}{\mu(\F_t(\omega))}d_X(Z_t(\omega),Z_{t+1}(a))^p\\
\le \sum_{a\in \F_t(\omega)}  \frac{\mu(a)}{\mu(\F_t(\omega))}d_X(z,Z_{t+1}(a))^p,
\end{multline}
where we used the martingale assumption $Z_t=\B(Z_{t+1}|\F_t)$. Let $A_1,\ldots,A_k\in \F_t$ be the distinct atoms of $\F_t$ and fix $\omega_i\in A_i$ for every $i\in \{1,\ldots,k\}$ (thus $A_i=\F_t(\omega_i)$). It follows from~\eqref{eq:use p bary conditioned} that for every $i\in \{1,\ldots,k\}$,
\begin{multline}\label{eq:on each i}
\sum_{a\in A_i} \mu(a)d_X(Z_t(\omega_i),Z_{t+1}(a))^p\\\le K^p\left(\sum_{a\in A_i}  \mu(a)d_X(z,Z_{t+1}(a))^p-\mu(A_i)d_X(Z_{t}(\omega_i),z)^p\right).
\end{multline}
Since $Z_t$ is constant on each of the sets $A_1,\ldots,A_k$, by summing~\eqref{eq:on each i} over $i\in \{1,\ldots,k\}$ we obtain the estimate
\begin{equation}\label{eq:to telescope}
\frac{1}{K^p}\int_\Omega d_X(Z_{t+1},Z_t)^pd\mu\le \int_\Omega d_X(Z_{t+1},z)^pd\mu-\int_\Omega d_X(Z_t,z)^pd\mu.
\end{equation}
The desired inequality~\eqref{eq:pisier bary} now follows since the right hand side of~\eqref{eq:to telescope}  telescopes upon summation  over $t\in \{0,\ldots,m-1\}$.
\end{proof}

\section{Proof of Theorem~\ref{thm:bary-cotype}}\label{sec:cotype proof}

Throughout the remainder of this paper it will be convenient to use
the following notation for Ces\'aro averages.
\begin{equation}\label{eq:cesaro notation}
\A_t(A)\eqdef \frac{1}{t} \sum_{s=1}^{t}A^s.
\end{equation}
The following simple lemma will be used in the proof of
Theorem~\ref{thm:bary-cotype}, as well as in Section~\ref{sec:ball's
cotype}.

\begin{lemma}\label{lem:control by cesaro}
Fix $p\in [1,\infty)$ and $n,t\in \N$. Suppose that  $A\in M_n(\R)$
is a stochastic matrix that is reversible relative to $\pi\in
\Delta^{n-1}$. Then for every metric space $(X,d_X)$ and every
$x_1,\ldots,x_n\in X$,
$$
\sum_{i=1}^n\sum_{j=1}^n \pi_i(A^t)_{i j}d_X(x_i,x_j)^p\le 2^p \sum_{i=1}^n\sum_{j=1}^n \pi_i\A_t(A)_{ij}d_X(x_i,x_j)^p.
$$
\end{lemma}

\begin{proof}
By the triangle inequality, for every $i,j,k\in \{1,\ldots,n\}$ we
have \begin{equation}\label{eq:p trianfle ijk}
d_X(x_i,x_j)^p\le
2^{p-1}\big(d_X(x_i,x_k)^p+d_X(x_k,x_j)^p\big).
\end{equation}
Consequently,
\begin{align*}
\nonumber&\sum_{i=1}^n\sum_{j=1}^n \pi_i(A^t)_{i j}d_X(x_i,x_j)^p\\\nonumber&=\frac{1}{t}\sum_{s=1}^t\sum_{i=1}^n\sum_{j=1}^n \pi_i\left(\sum_{k=1}^n (A^s)_{i k}(A^{t-s})_{kj}\right) d_X(x_i,x_j)^p\\\nonumber
&\stackrel{\eqref{eq:p trianfle ijk}}{\le} \frac{2^{p-1}}{t}\sum_{s=1}^t\sum_{i=1}^n\sum_{j=1}^n\sum_{k=1}^n \pi_i(A^s)_{i k}(A^{t-s})_{kj}
\big(d_X(x_i,x_k)^p+d_X(x_k,x_j)^p\big)\\
\nonumber&\stackrel{\eqref{eq:cesaro notation}}{=} 2^{p-1}\sum_{i=1}^n\sum_{j=1}^n \pi_i
\left(\A_t(A)_{ij}+\frac{t-1}{t}\A_{t-1}(A)_{ij}\right)d_X(x_i,x_j)^p\\
&\le 2^p \sum_{i=1}^n\sum_{j=1}^n \pi_i\A_t(A)_{ij}d_X(x_i,x_j)^p.\qedhere
\end{align*}
\end{proof}

\begin{proof}[Proof of Theorem~\ref{thm:bary-cotype}]
Fix $p,K,\Gamma\in [1,\infty)$, a metric space $(X,d_X)$ and a
barycenter map $\B:\P_X\to X$ with respect to which $(X,d_X)$ is
both $W_p$ barycentric with constant $\Gamma$ and $p$-barycentric
with constant $K$. We also fix $n,t\in \N$, a probability vector
$\pi\in \Delta^{n-1}$ and a stochastic matrix $A=(a_{ij})$ that is
reversible relative to $\pi$. Given $x_1,\ldots,x_n\in X$ our goal
is to prove that there exist $y_1,\ldots,y_n\in X$ such
that~\eqref{def:markov cotype} is satisfied with $N\lesssim \Gamma
K$.

In the proof of Theorem~\ref{thm:bary-cotype} we may assume that the right hand side of~\eqref{def:markov cotype} is nonzero. By restricting to the support of $\pi$ we may also assume that $\pi\in (0,1)^n$. Letting $\Pi\in M_n(\R)$ be given by $\Pi_{ij}=\pi_j$, choose $\e\in (0,1/2)$ small enough so that for $B=(1-\e)A+\e \Pi$ we have
$$
\sum_{i=1}^n\sum_{j=1}^n \pi_i \A_t(B)_{ij} d_X(x_i,x_j)^p\lesssim \sum_{i=1}^n\sum_{j=1}^n \pi_i \A_t(A)_{ij} d_X(x_i,x_j)^p.
$$
Since $B_{ij}\ge (1-\e)a_{ij}\ge \frac12 a_{ij}$, for every $y_1,\ldots,y_n\in X$ we have
\begin{multline*}
\sum_{i=1}^n \pi_id_X(x_i,y_i)^p+t\sum_{i=1}^n\sum_{j=1}^n \pi_iB_{ij}
d_X(y_i,y_j)^p\\\gtrsim \sum_{i=1}^n \pi_id_X(x_i,y_i)^p+t\sum_{i=1}^n\sum_{j=1}^n \pi_ia_{ij}
d_X(y_i,y_j)^p.
\end{multline*}
Since the matrix $B$ is stochastic and reversible relative to $\pi$ and none of its entries vanish, this shows that it suffices to prove Theorem~\ref{thm:bary-cotype} under the assumption that $\pi_i,a_{ij}>0$ for every $i,j\in \{1,\ldots,n\}$.

Denote $\Omega=\{1,\ldots,n\}^t$. Set $\F_0=\{\emptyset,\Omega\}$ and for every $s\in \{1,\ldots,t\}$ let $\F_s\subseteq 2^\Omega$ be the $\sigma$-algebra generated by the first $s$ coordinates, i.e., the atoms of $\F_s$ are $\{E_{\tau}\}_{\tau\in \{1,\ldots,n\}^s}$, where we denote
for every $(i_1,\ldots,i_s)\in \{1,\ldots,n\}^s$,
$$
E_{(i_1,\ldots,i_s)}\eqdef \big\{(j_1,\ldots,j_t)\in \Omega:\ (j_1,\ldots,j_s)=(i_1,\ldots,i_s)\big\}.
$$

Fix $\ell\in \{1,\ldots,n\}$ and define  $\mu_\ell\in \P_\Omega$ on $\Omega$ by
$$\forall\, (i_1,\ldots,i_t)\in \Omega,\qquad \mu_\ell(i_1,\ldots,i_t)\eqdef a_{\ell,i_1}\prod_{s=1}^{t-1}a_{i_s,i_{s+1}}.$$
Thus $(\Omega,\mu_\ell)$ is the probability space of trajectories of length $t$ of the Markov chain on $\{1,\ldots,n\}$ that starts at $\ell$ and whose transition matrix is $A$. By definition, $\mu_\ell(\omega)>0$ for every $\omega\in \Omega$.

We next define inductively mappings $M_0^{(\ell,t)},\ldots,M_t^{(\ell,t)}:\Omega\to X$ as follows. For every $(i_1,\ldots,i_t)\in \Omega$,
$$
M_t^{(\ell,t)}(i_1,\ldots,i_t)\eqdef x_{i_t},
$$
and for every $s\in \{0,\ldots,t-1\}$,
\begin{equation}\label{eq:def ell martingale}
M_s^{(\ell,t)}(i_1,\ldots,i_t)\eqdef \B\left(\sum_{j=1}^n
a_{i_s,j}\cdot \d_{M_{s+1}^{(\ell,t)}(i_1,\ldots,i_s,j,1\ldots,1)}\right),
\end{equation}
where $i_0=\ell$. Thus $M_s^{(\ell,t)}(i_1,\ldots,i_t)$ depends only
on $(i_1,\ldots,i_s)$. We may therefore think of $M_s^{(\ell,t)}$ as
an $X$-valued function defined on $\{1,\ldots,n\}^s$, and
$M_0^{(\ell,t)}$ as a point in $X$.

Definition~\eqref{eq:def ell martingale} implies that for $s\in \{0,\ldots,t\}$ and $(i_1,\ldots,i_t)\in \Omega$,
\begin{equation}\label{eq:last coordinate dependence}
M_s^{(\ell,t)}(i_1,\ldots,i_t)=M_0^{(i_s,t-s)}.
\end{equation}
Moreover, recalling~\eqref{eq:def conditional} it follows from~\eqref{eq:def ell martingale} that for $s\in \{0,\ldots,t-1\}$,
$$
M_s^{(\ell,t)}=\B\left(\left. M_{s+1}^{(\ell,t)}\right|\F_s\right).
$$
 Thus $M_0^{(\ell,t)},\ldots,M_t^{(\ell,t)}$ is an $X$-valued martingale with respect to the measure $\mu_\ell$ and the filtration $\F_0\subseteq \ldots,\subseteq \F_t=2^\Omega$. An application of Lemma~\ref{lem:pisier} (with $z=x_\ell$) therefore yields the following estimate.
\begin{multline}\label{eq:use pisier}
\sum_{s=1}^{t} \sum_{i=1}^n (A^{s-1})_{\ell,i}\sum_{j=1}^n a_{ij} d_X\left(M_0^{(i,t-s+1)},M_0^{(j,t-s)}\right)^p\\\le K^p\left(\sum_{j=1}^n(A^t)_{\ell j}d_X\left(x_j,x_\ell\right)^p-d_X\left(x_\ell,M_0^{(\ell,t)}\right)^p\right).
\end{multline}

Multiplying~\eqref{eq:use pisier} by $\pi_\ell$ and summing over
$\ell\in \{1,\ldots, n\}$ while using the fact that $A^{s-1}$ is
stochastic and reversible relative to $\pi$ shows that
\begin{multline}\label{eq:summed over ell}
\sum_{s=1}^t\sum_{i=1}^n\sum_{j=1}^n \pi_i a_{ij}  d_X\left(M_0^{(i,s)},M_0^{(j,s-1)}\right)^p\\\le K^p\left(\sum_{\ell=1}^n\sum_{j=1}^n \pi_\ell (A^t)_{\ell j}d_X(x_\ell,x_j)^p-\sum_{\ell=1}^n \pi_\ell d_X\left(x_\ell,M_0^{(\ell,t)}\right)^p\ \right).
\end{multline}

In order to bound the left hand side of~\eqref{eq:summed over ell} from below, observe that for every $i,j\in \{1,\ldots,n\}$ condition~\eqref{eq:Gamma Wp Lip} of our assumption that $(X,d_X)$ is $W_p$ barycentric with constant $\Gamma$ implies that
\begin{multline}\label{eq:just the s sum}
\sum_{s=1}^t d_X\left(M_0^{(i,s)},M_0^{(j,s-1)}\right)^p\\\ge \frac{t}{\Gamma^p}d_X\left(\B\left(\frac{1}{t}\sum_{s=1}^t \d_{M_0^{(i,s)}}\right),\B\left(\frac{1}{t}\sum_{s=1}^t \d_{M_0^{(j,s-1)}}\right)\right)^{p}.
\end{multline}
Moreover, by the triangle inequality and convexity of $u\mapsto u^p$ on $[0,\infty)$,
\begin{align}\label{eq:triangle barys}
\nonumber&d_X\left(\B\left(\frac{1}{t}\sum_{s=1}^t \d_{M_0^{(i,s)}}\right),\B\left(\frac{1}{t}\sum_{s=1}^t \d_{M_0^{(j,s)}}\right)\right)^{p}\\
\nonumber&\le 2^{p-1} d_X\left(\B\left(\frac{1}{t}\sum_{s=1}^t \d_{M_0^{(i,s)}}\right),\B\left(\frac{1}{t}\sum_{s=1}^t \d_{M_0^{(j,s-1)}}\right)\right)^{p}\\&\quad+2^{p-1}d_X\left(\B\left(\frac{1}{t}\sum_{s=1}^t \d_{M_0^{(j,s)}}\right),\B\left(\frac{1}{t}\sum_{s=1}^t \d_{M_0^{(j,s-1)}}\right)\right)^{p}.
\end{align}
Another application of~\eqref{eq:Gamma Wp Lip} shows that
\begin{multline}\label{eq:second use Gamma}
d_X\left(\B\left(\frac{1}{t}\sum_{s=1}^t \d_{M_0^{(j,s)}}\right),\B\left(\frac{1}{t}\sum_{s=1}^t \d_{M_0^{(j,s-1)}}\right)\right)\\\le \frac{\Gamma}{t^{1/p}} d_X\left(M_0^{(j,t)},M_0^{(j,0)}\right)=\frac{\Gamma}{t^{1/p}} d_X\left(M_0^{(j,t)},x_j\right).
\end{multline}
Consequently, if we define
\begin{equation}\label{eq:def yi cotype}
y_i\eqdef \B\left(\frac{1}{t}\sum_{s=1}^t \d_{M_0^{(i,s)}}\right),
\end{equation}
then it follows from~\eqref{eq:just the s sum}, \eqref{eq:triangle barys} and~\eqref{eq:second use Gamma} that
\begin{multline}\label{eq:y_i appear}
\sum_{s=1}^t\sum_{i=1}^n\sum_{j=1}^n \pi_i a_{ij}  d_X\left(M_0^{(i,s)},M_0^{(j,s-1)}\right)^p\\\ge \frac{2t}{(2\Gamma)^p}\sum_{i=1}^n\sum_{j=1}^n \pi_i a_{ij} d_X(y_i,y_j)^p-\sum_{j=1}^n \pi_j d_X\left(M_0^{(j,t)},x_j\right)^p.
\end{multline}
A substitution of~\eqref{eq:y_i appear} into~\eqref{eq:summed over ell} now yields the following estimate.
\begin{align}\label{eq:t term proved}
\nonumber&t\sum_{i=1}^n\sum_{j=1}^n \pi_i a_{ij} d_X(y_i,y_j)^p\\& \nonumber\le
\frac{(2\Gamma K)^p}{2}\sum_{i=1}^n\sum_{j=1}^n \pi_i(A^t)_{i j}d_X(x_i,x_j)^p\\&\quad-
\nonumber2^{p-1}\Gamma^p\left(K^p-1\right)\sum_{j=1}^n \pi_j d_X\left(M_0^{(j,t)},x_j\right)^p\\
&\le \frac{(2\Gamma K)^p}{2}\sum_{i=1}^n\sum_{j=1}^n \pi_i(A^t)_{i j}d_X(x_i,x_j)^p.
\end{align}
By combining~\eqref{eq:t term proved} and Lemma~\ref{lem:control by
cesaro} we deduce that
\begin{equation}\label{eq;bring in cesaro}
t\sum_{i=1}^n\sum_{j=1}^n \pi_i a_{ij} d_X(y_i,y_j)^p\le (4\Gamma K)^p \sum_{i=1}^n\sum_{j=1}^n \pi_i\A_t(A)_{ij}d_X(x_i,x_j)^p.
\end{equation}

Next, we need to bound the quantity $\sum_{i=1}^n \pi_i
d_X(x_i,y_i)^p$. We first claim that for every $i\in \{1,\ldots,
n\}$, every $s\in \{0,\ldots,t\}$ and every $z\in X$ we have
\begin{equation}\label{eq:for iteration gamma=1}
d_X\left(z,M_0^{(i,t)}\right)^p\le \sum_{j=1}^n (A^s)_{ij} d_X\left(z,M_0^{(j,t-s)}\right)^p.
\end{equation}
The proof of~\eqref{eq:for iteration gamma=1} is by induction on $s$. For $s=0$ the desired  inequality~\eqref{eq:for iteration gamma=1} holds as equality. Assuming the validity of~\eqref{eq:for iteration gamma=1} for some $s\in \{0,\ldots,t-1\}$, and recalling~\eqref{eq:def ell martingale} and~\eqref{eq:last coordinate dependence}, observe that for every $j\in \{1,\ldots,n\}$ we have
$$
M_0^{(j,t-s)}=\B\left(\sum_{k=1}^n a_{jk} \d_{M_0^{(k,t-s+1)}}\right).
$$
Consequently, it follows from~\eqref{eq:def p bary} that
\begin{equation}\label{eq:use p bary without second term}
d_X\left(z,M_0^{(j,t-s)}\right)^p\le \sum_{k=1}^n a_{jk} d_X\left(z,M_0^{(k,t-s+1)}\right)^p.
\end{equation}
So,
\begin{multline*}
d_X\left(z,M_0^{(i,t)}\right)^p
\stackrel{\eqref{eq:for iteration gamma=1}\wedge\eqref{eq:use p bary without second term}}{\le} \sum_{j=1}^n (A^s)_{ij}\sum_{k=1}^n a_{jk} d_X\left(z,M_0^{(k,t-s+1)}\right)^p\\=\sum_{k=1}^n (A^{s+1})_{ik} d_X\left(z,M_0^{(k,t-s+1)}\right)^p,
\end{multline*}
thus completing the inductive verification of~\eqref{eq:for iteration gamma=1}.

When $s=t$ inequality~\eqref{eq:for iteration gamma=1} becomes
\begin{equation}\label{eq:post iteration s=t}
d_X\left(z,M_0^{(i,t)}\right)^p\le \sum_{j=1}^n (A^t)_{ij} d_X\left(z,x_j\right)^p.
\end{equation}
Hence, for every $i\in \{1,\ldots,n\}$,
\begin{multline*}
d_X(x_i,y_i)^p\stackrel{\eqref{eq:def yi cotype}}{=} d_X\left(x_i,\B\left(\frac{1}{t}\sum_{s=1}^t \d_{M_0^{(i,s)}}\right)\right)^p\stackrel{\eqref{eq:def p bary}}{\le} \frac{1}{t} \sum_{s=1}^t d_X\left(x_i, M_0^{(i,s)}\right)^p\\
\stackrel{\eqref{eq:post iteration s=t}}{\le} \frac{1}{t}\sum_{s=1}^t \sum_{j=1}^n (A^s)_{ij} d_X\left(x_i,x_j\right)^p\stackrel{\eqref{eq:cesaro notation}}{=}\sum_{j=1}^n \A_t(A)_{ij} d_X\left(x_i,x_j\right)^p.
\end{multline*}
Consequently,
\begin{equation}\label{eq:diagonal sum}
\sum_{i=1}^n\pi_i d_X(x_i,y_i)^p\le \sum_{i=1}^n\sum_{j=1}^n \pi_i\A_t(A)_{ij} d_X\left(x_i,x_j\right)^p.
\end{equation}
By summing~\eqref{eq;bring in cesaro} and~\eqref{eq:diagonal sum} we
conclude that the desired inequality~\eqref{def:markov cotype} holds
true with $y_1,\ldots,y_n$ chosen as in~\eqref{eq:def yi cotype} and
$N^p=(4\Gamma K)^p+1$. This completes the proof of
Theorem~\ref{thm:bary-cotype}.\end{proof}

\section{Proof of Theorem~\ref{thm:calculus gap}}\label{sec: norm}

Theorem~\ref{thm:calculus gap} is a consequence of Lemma~\ref{lem:even t iterated version} and Lemma~\ref{lem:bound lambda by gamma} below. These lemmas are meaningful without the restriction $K=1$ of Theorem~\ref{thm:calculus gap}: this more stringent assumption will only be used later, when we combine Lemma~\ref{lem:even t iterated version} and Lemma~\ref{lem:bound lambda by gamma} to deduce Theorem~\ref{thm:calculus gap}.

In order to state our results we need to first introduce a small amount of notation. Given a metric space $(X,d_X)$ and $p\in [1,\infty)$, for every $n\in \N$  we denote by $L_p^n(X)$ the space of all function $f:\{1,\ldots, n\}\to X$, equipped with the metric
\begin{equation*}\label{eq:L_p metric}
\forall\, f,g\in L_p^n(X),\qquad d_{L_p^n(X)}(f,g)\eqdef \left(\frac{1}{n}\sum_{i=1}^n d_X(f(i),g(i))^p\right)^{1/p}.
\end{equation*}

Suppose that $\B:\P_X\to X$ is a barycenter map. In what follows it will be convenient to use the following slight abuse of notation: for every $f:\{1,\ldots,n\}\to X$ write
$$
\B(f)\eqdef \B\left(\frac{1}{n}\sum_{i=1}^n \d_{f(i)}\right).
$$

For a symmetric stochastic matrix $A=(a_{ij})\in M_n(\R)$ define a
mapping $A\otimes I_X^n: L_p^n(X)\to L_p^n(X)$ by setting for every
$i\in \{1,\ldots n\}$ and $f:\{1,\ldots,n\}\to X$,
$$
(A\otimes I_X^n)(f)(i)\eqdef \B\left(\sum_{j=1}^n a_{ij} \d_{f(j)}\right).
$$
We warn that, unlike in the setting of Banach space valued mappings, given two symmetric stochastic matrices $A,B\in M_n(\R)$ the composition $(A\otimes I_X^n)\circ (B\otimes I_X^n)$ need not be of the form $C\otimes I_X^n$ for some symmetric stochastic matrix $C\in M_n(\R)$, and in particular the identity $(A\otimes I_X^n)\circ (B\otimes I_X^n)=(AB)\otimes I_X$ need not hold true.

\begin{definition}\label{def:lambda}
Fix $p\in [1,\infty)$ and a metric space $(X,d_X)$ equipped with a barycenter map $\B:\P_X\to X$. Given $T:L_p^n(X)\to L_p^n(X)$, define $\lambda_p(T)\in [0,\infty]$ to be the infimum over those $\lambda\in (0,\infty]$ for which every $f\in L_p^n(X)$ satisfies
$$
d_{L_p^n(X)}\left(T(f),\B\left(T(f)\right)\right)\le \lambda \cdot d_{L_p^n(X)}\left(f,\B(f)\right).
$$
\end{definition}

Lemma~\ref{lem:even t iterated version} below relates the quantities $\gamma_+(A^t,d_X^p)$ and $\lambda_p\left(A\otimes I_X^n\right)$. Note that it assumes that $(X,d_X)$ is $p$-barycentric with constant $K$, but $K$ does not appear in the conclusion~\eqref{eq:Gamma but no K}.  The reason for this is that the proof of Lemma~\ref{lem:even t iterated version} uses a weaker version of~\eqref{eq:def p bary} in which the rightmost term on the left hand side of~\eqref{eq:def p bary} is dropped, i.e., the assumption that $(X,d_X)$ is $p$-barycentric with constant $K$ is used in Lemma~\ref{lem:even t iterated version} only through the requirement that every $\mu\in \P_X$ satisfies
$$
\forall\, x\in X,\qquad d_X(x,\B(\mu))^p\le \int_{X}d_X(x,y)^pd\mu(y).
$$
\begin{lemma}\label{lem:even t iterated version}
Fix $p,K,\Gamma\in [1,\infty)$ and $n,t\in \N$. Suppose that
$(X,d_X)$ is a metric space that is both $W_p$ barycentric with
constant $\Gamma$ and $p$-barycentric with constant $K$. Let
$A=(a_{ij})\in M_n(\R)$ be  a symmetric stochastic matrix such that
$\lambda_p\left(A\otimes I_X^n\right)<1$. Then
\begin{equation}\label{eq:Gamma but no K}
\gamma_+(A^t,d_X^p)\le \left(\Gamma+ \frac{4(\Gamma+1)}{1-\lambda_p\left(A\otimes I_X^n\right)^{2t}} \right)^p.
\end{equation}
\end{lemma}

In what follows, for $p,K\in [1,\infty)$ we denote by $\beta_p(K)$ the unique $\beta\in [1,\infty)$ satisfying
\begin{equation}\label{eq:def beta}
\beta^p+K^p(\beta-1)^p=K^p.
\end{equation}
Thus in particular,
$$
\beta_2(K)=\frac{2K^2}{K^2+1}.
$$
Observe that $\beta_p(K)\in [1,\min\{K,2\}]$ and $\beta_p(K)=1$ if and only if $K=1$. We also have
\begin{equation}\label{max formulation beta}
\beta_p(K)=\max_{\beta\in [1,2]} \min\left\{\beta, K\left(1-(\beta-1)^p\right)^{1/p}\right\}.
\end{equation}
To verify~\eqref{max formulation beta} note that the function $\beta\mapsto K\left(1-(\beta-1)^p\right)^{1/p}$ decreases from $K$ to $0$ on $[1,2]$, so the maximum that appears in~\eqref{max formulation beta} is attained when $\beta= K\left(1-(\beta-1)^p\right)^{1/p}$, or, due to~\eqref{eq:def beta}, when $\beta=\beta_p(K)$. An equivalent way to state~\eqref{max formulation beta} is that for every $a\in [0,\infty)$ we have
\begin{equation}\label{eq:ab}
0\le b\le a\implies \min\left\{a+b, K\left(a^p-b^p\right)^{1/p}\right\}\le \beta_p(K)a.
\end{equation}
To deduce~\eqref{eq:ab} from~\eqref{max formulation beta} simply write $b=(\beta-1)a$ for some $\beta\in [1,2]$.

\begin{lemma}\label{lem:bound lambda by gamma}Fix $p,K\in [1,\infty)$ and $n\in \N$.
Suppose that $(X,d_X)$ is a $p$-barycentric metric space with
constant $K$. Then every symmetric stochastic matrix $A=(a_{ij})\in
M_n(\R)$ satisfies
$$\lambda_p\left(A\otimes I_X^n\right)\le \beta_p(K)\left( \frac{K^{2p}\gamma_+(A,d_X^p)-1}{K^{2p}\gamma_+(A,d_X^p)+K^p}\right)^{1/p}.$$
\end{lemma}

Assuming the validity of Lemma~\ref{lem:even t iterated version} and Lemma~\ref{lem:bound lambda by gamma} for the moment, we now show how they imply Theorem~\ref{thm:calculus gap}.

\begin{proof}[Proof of Theorem~\ref{thm:calculus gap}]
We are now assuming that $(X,d_X)$ is both $W_p$ barycentric with constant $\Gamma$ and $p$-barycentric with constant $K=1$. Under the latter assumption the conclusion of Lemma~\ref{lem:bound lambda by gamma} becomes
\begin{equation}\label{eq:K=1 version}
\lambda_p\left(A\otimes I_X^n\right)\le \left( \frac{\gamma_+(A,d_X^p)-1}{\gamma_+(A,d_X^p)+1}\right)^{1/p}.
\end{equation}
Thus in particular $\lambda_p\left(A\otimes I_X^n\right)<1$, so we may use Lemma~\ref{lem:even t iterated version} in conjunction with~\eqref{eq:K=1 version} to obtain the estimate
$$
\gamma_+(A^t,d_X^p)^{1/p}\lesssim \frac{\Gamma}{1-\left( \frac{\gamma_+(A,d_X^p)-1}{\gamma_+(A,d_X^p)+1}\right)^{2t/p}}.
$$
It now remains to note the elementary inequality
$$
\forall\, p,\gamma,t\in [1,\infty),\qquad 1-\left( \frac{\gamma-1}{\gamma+1}\right)^{2t/p}\gtrsim \min\left\{1,\frac{t}{p\gamma}\right\},
$$
which follows by considering the cases $t\ge p\gamma$ and $t\le p\gamma$ separately.
\end{proof}

We now proceed to prove Lemma~\ref{lem:even t iterated version} and Lemma~\ref{lem:bound lambda by gamma}. To this end, given $n,t\in \N$ and $T:L_p^n(X)\to L_p^n(X)$, we denote the $t$-fold iterate of $T$ by $T^{[t]}$, i.e.,
\begin{equation*}
T^{[t]}\eqdef \underbrace{T\circ \ldots \circ T}_{t\ \mathrm{times}}.
\end{equation*}
We also use the convention that $T^{[0]}$ is the identity mapping.  If $X$ is a Banach space then $(A\otimes I_X^n)^{[t]}=A^t\otimes I_X^n$, but this need not hold true when $X$ is not a Banach space. Observe that a direct iterative application of Definition~\ref{def:lambda} implies that
\begin{equation}\label{eq:lambda decay}
\lambda_p\left(T^{[t]}\right)\le \lambda_p(T)^t.
\end{equation}

\begin{lemma}\label{lem:f close to A iterate}
Fix $p,K\in [1,\infty)$. Suppose that $(X,d_X)$ is a $p$-barycentric
metric space with constant $K$. Then for every $n,t\in \N$, every
symmetric stochastic matrix $A=(a_{ij})\in M_n(\R)$ and every $f\in
L_p^n(X)$,
\begin{equation}\label{eq:distance from iterate}
d_{L_p^n(X)}\left(f,(A\otimes I_X^n)^{[t]}(f)\right)^p\le \frac{1}{n}\sum_{i=1}^n\sum_{j=1}^n (A^t)_{ij} d_X(f(i),f(j))^p.
\end{equation}
\end{lemma}

\begin{proof}
We will prove by induction on $t$ that if $B=(b_{ij})\in M_n(\R)$ has nonnegative entries then
\begin{multline}\label{eq:B version}
\sum_{i=1}^n\sum_{j=1}^n b_{ij} d_X\left(f(i),(A\otimes I_X^n)^{[t]}(f)(j)\right)^p\\\le \sum_{i=1}^n\sum_{j=1}^n (BA^t)_{ij} d_X(f(i),f(j))^p.
\end{multline}
The desired inequality~\eqref{eq:distance from iterate} is then the special case of~\eqref{eq:B version} when $B$ is the identity matrix.

\eqref{eq:B version} holds as equality when $t=0$, so assume inductively that~\eqref{eq:B version} holds true for some $t\in \N\cup \{0\}$. Fix $i,j\in \{1,\ldots,n\}$ and consider the probability measure $\mu_j\in \P_X$ given  by
$$
\mu_j\eqdef \sum_{k=1}^n a_{jk} \d_{(A\otimes I_X^n)^{[t]}(f)(k)}.
$$
Then $\B(\mu_j)=(A\otimes I_X^n)^{[t+1]}(f)(j)$. By~\eqref{eq:def p
bary} with $\mu=\mu_j$ and $x=f(i)$,
\begin{multline}\label{use bary ij}
d_X\left(f(i),(A\otimes I_X^n)^{[t+1]}(f)(j)\right)^p\\\le \sum_{k=1}^n a_{jk} d_X\left(f(i),(A\otimes I_X^n)^{[t]}(f)(k)\right)^p.
\end{multline}
Hence,
\begin{align}
\nonumber&\sum_{i=1}^n\sum_{j=1}^n b_{ij} d_X\left(f(i),(A\otimes I_X^n)^{[t+1]}(f)(j)\right)^p
\\\label{eq:use ij B}&\le \sum_{i=1}^n\sum_{j=1}^n b_{ij}\sum_{k=1}^n a_{jk} d_X\left(f(i),(A\otimes I_X^n)^{[t]}(f)(k)\right)^p
\\\nonumber&=\sum_{i=1}^n\sum_{k=1}^n (BA)_{ik} d_X\left(f(i),(A\otimes I_X^n)^{[t]}(f)(k)\right)^p\\
\label{eq:use indunctive B}&\le \sum_{i=1}^n\sum_{k=1}^n (BA^{t+1})_{ik} d_X(f(i),f(k))^p,
\end{align}
where in~\eqref{eq:use ij B} uses~\eqref{use bary ij} and~\eqref{eq:use indunctive B} uses the inductive hypothesis.
\end{proof}

\begin{lemma}\label{lem:Gamma+1}
Fix $p,\Gamma\in [1,\infty)$. Suppose that $(X,d_X)$ is a $W_p$ barycentric metric space with constant $\Gamma$. Then for every $f,g\in L_p^n(X)$,
\begin{equation}\label{eq:Gamma+1}
\left|d_{L_p^n(X)} (f,\B(f))-d_{L_p^n(X)} (g,\B(g))\right|\le (\Gamma+1)d_{L_p^n(X)} (f,g).
\end{equation}
\end{lemma}

\begin{proof} By~\eqref{eq:Gamma Wp Lip} we have
\begin{equation}\label{eq:use wp lip fg}
d_X(\B(f),\B(g))\le \Gamma\cdot d_{L_p^n(X)}(f,g).
\end{equation}
By the triangle inequality in $L_p^n(\R)$ we have
\begin{multline}\label{eq:ellpn triangle}
\left|d_{L_p^n(X)} (f,\B(f))-d_{L_p^n(X)} (g,\B(g))\right|\\\le \left(\frac{1}{n}\sum_{i=1}^n \left|d_X(f(i),\B(f))-d_X(g(i),\B(g))\right|^p\right)^{1/p}.
\end{multline}
For every $i\in \{1,\ldots, n\}$ the triangle inequality in $(X,d_X)$ implies that
$$
\left|d_X(f(i),\B(f))-d_X(g(i),\B(g))\right|\le d_X(f(i),g(i))+d_X(\B(f),\B(g)).
$$
In combination with~\eqref{eq:ellpn triangle} and another application of the triangle inequality in $L_p^n(\R)$ we deduce that
$$
\left|d_{L_p^n(X)} (f,\B(f))-d_{L_p^n(X)} (g,\B(g))\right|\le d_{L_p^n(X)}(f,g)+d_X(\B(f),\B(g)).
$$
Due to~\eqref{eq:use wp lip fg}, this implies the desired inequality~\eqref{eq:Gamma+1}.
\end{proof}

\begin{lemma}\label{lem:even t}
Fix $p,K,\Gamma\in [1,\infty)$ and $n,t\in \N$. Suppose that $(X,d_X)$ is a metric space that is both $W_p$ barycentric with constant $\Gamma$ and $p$-barycentric with constant $K$.  Let $A\in M_n(\R)$ be  a symmetric stochastic matrix such that $\lambda_p\left((A\otimes I_X^n)^{[2t]}\right)<1$. Then
\begin{equation}\label{eq:gamma+ upper in terms of lambda}
\gamma_+(A^t,d_X^p)\le \left(\Gamma+ \frac{4(\Gamma+1)}{1-\lambda_p\left((A\otimes I_X^n)^{[2t]}\right)} \right)^p.
\end{equation}
\end{lemma}
\begin{proof}
For every $f,g:\{1,\ldots,n\}\to X$ we have
\begin{multline}\label{eq:three terms n2}
\left(\frac{1}{n^2}\sum_{i=1}^n\sum_{j=1}^n d_X(f(i),g(j))^p\right)^{1/p}\\\le d_X(\B(f),\B(g))+d_{L_p^n(X)}(f,\B(f))+d_{L_p^n(X)}(g,\B(g)).
\end{multline}
We proceed to bound each of the terms on the right hand side of~\eqref{eq:three terms n2} separately.

First, define $\mu_f,\mu_g\in \P_X$ by
$$
\mu_f\eqdef \frac{1}{n}\sum_{i=1}^n \d_{f(i)}\qquad\mathrm{and}\qquad \mu_g\eqdef \frac{1}{n}\sum_{i=1}^n \d_{g(i)}.
$$
Then $\B(f)=\B(\mu_f)$ and $\B(g)=\B(\mu_g)$. If we define $\pi\in \P_{X\times X}$ by
$$
\pi=\frac{1}{n}\sum_{i=1}^n\sum_{j=1}^n (A^t)_{ij} \d_{(f(i),g(j))},
$$
then, since $A^t$ is a symmetric stochastic matrix,  $\pi\in \Pi(\mu_f,\mu_g)$, i.e., $\pi$ is a coupling of $\mu_f$ and $\mu_g$. It therefore follows from~\eqref{eq:Gamma Wp Lip} that
\begin{equation}\label{eq:estimated first of three}
d_X(\B(f),\B(g))\le \Gamma\left(\frac{1}{n}\sum_{i=1}^n\sum_{j=1}^n (A^t)_{ij} d_X(f(i),g(j))^p\right)^{1/p}.
\end{equation}

Next, since
\begin{multline*}
d_{L_p^n(X)}\left((A\otimes I_X^n)^{[2t]}(f),\B\left((A\otimes I_X^n)^{[2t]}(f)\right)\right)\\\le \lambda_p\left((A\otimes I_X^n)^{[2t]}\right) d_{L_p^n(X)}(f,\B(f)),
\end{multline*}
and $\lambda_p\left((A\otimes I_X^n)^{[2t]}\right)<1$, we have
\begin{multline}\label{eq:divide 1-lambda}
 \frac{d_{L_p^n(X)}(f,\B(f))-d_{L_p^n(X)}\left((A\otimes I_X^n)^{[2t]}(f),\B\left((A\otimes I_X^n)^{[2t]}(f)\right)\right)}{1-\lambda_p\left((A\otimes I_X^n)^{[2t]}\right)}\\\ge d_{L_p^n(X)}(f,\B(f)).
\end{multline}
Now,
\begin{align}
&\nonumber d_{L_p^n(X)}(f,\B(f))-d_{L_p^n(X)}\left((A\otimes I_X^n)^{[2t]}(f),\B\left((A\otimes I_X^n)^{[2t]}(f)\right)\right)\\\label{eq:use lemma gamma+1}&\le (\Gamma+1) d_{L_p^n(X)}\left(f,(A\otimes I_X^n)^{[2t]}(f)\right)\\
\label{eq:use lemma dist to iteration}&\le (\Gamma+1)\left(\frac{1}{n}\sum_{i=1}^n\sum_{j=1}^n (A^{2t})_{ij} d_X(f(i),f(j))^p\right)^{1/p}.
\end{align}
where~\eqref{eq:use lemma gamma+1} uses Lemma~\ref{lem:Gamma+1} and~\eqref{eq:use lemma dist to iteration} uses Lemma~\ref{lem:f close to A iterate}.

Observe that
\begin{align}\label{eq:2t=t+t}
\nonumber&\sum_{i=1}^n\sum_{j=1}^n (A^{2t})_{ij} d_X(f(i),f(j))^p\\\nonumber&=\sum_{i=1}^n\sum_{j=1}^n \left(\sum_{k=1}^n (A^t)_{ik}(A^t)_{kj}\right)d_X(f(i),f(j))^p\\
\nonumber&\le 2^{p-1}\sum_{i=1}^n\sum_{j=1}^n\sum_{k=1}^n (A^t)_{ik}(A^t)_{kj}\big(d_X(f(i),g(k))^p+d_X(g(k),f(j))^p\big)\\
&= 2^p\sum_{i=1}^n\sum_{j=1}^n (A^t)_{ij} d_X(f(i),g(j))^p.
\end{align}
By combining~\eqref{eq:divide 1-lambda} with~\eqref{eq:use lemma dist to iteration} and~\eqref{eq:2t=t+t} we have,
\begin{multline}\label{eq:d f Bf}
d_{L_p^n(X)}(f,\B(f))\\\le \frac{2(\Gamma+1)}{1-\lambda_p\left((A\otimes I_X^n)^{[2t]}\right)} \left(\frac{1}{n}\sum_{i=1}^n\sum_{j=1}^n (A^t)_{ij} d_X(f(i),g(j))^p\right)^{1/p}.
\end{multline}
By the same reasoning,
\begin{multline}\label{eq:d g Bg}
d_{L_p^n(X)}(g,\B(g))\\\le \frac{2(\Gamma+1)}{1-\lambda_p\left((A\otimes I_X^n)^{[2t]}\right)} \left(\frac{1}{n}\sum_{i=1}^n\sum_{j=1}^n (A^t)_{ij} d_X(f(i),g(j))^p\right)^{1/p}.
\end{multline}

Finally, a substitution of~\eqref{eq:estimated first of three}, \eqref{eq:d f Bf} and~\eqref{eq:d g Bg} into~\eqref{eq:three terms n2} shows that
\begin{multline}\label{eq:for def gamma+}
\frac{1}{n^2}\sum_{i=1}^n\sum_{j=1}^n d_X(f(i),g(j))^p\\\le\left(\Gamma+ \frac{4(\Gamma+1)}{1-\lambda_p\left((A\otimes I_X^n)^{[2t]}\right)} \right)^p\cdot\frac{1}{n}\sum_{i=1}^n\sum_{j=1}^n (A^t)_{ij} d_X(f(i),g(j))^p.
\end{multline}
Since~\eqref{eq:for def gamma+} holds true for every $f,g:\{1,\ldots,n\}\to X$, the desired estimate~\eqref{eq:gamma+ upper in terms of lambda} now follows by recalling the definition of $\gamma_+(A,d_X^p)$.
\end{proof}

\begin{proof}[Proof of Lemma~\ref{lem:even t iterated version}] The desired estimate~\eqref{eq:Gamma but no K} is a consequence of~\eqref{eq:lambda decay} and Lemma~\ref{lem:even t}.
\end{proof}

We now proceed to prove Lemma~\ref{lem:bound lambda by gamma}. Recalling the definition of $\beta_p(K)$ in~\eqref{eq:def beta}, we first establish the following estimate.

\begin{lemma}\label{lem:use beta}
Fix $p,K\in [1,\infty)$. Suppose that $(X,d_X)$ is a $p$-barycentric metric space with constant $K$. Then for every $n\in \N$, every symmetric stochastic matrix $A\in M_n(\R)$ and every $f\in L_p^n(X)$ we have
\begin{multline}\label{eq:beta ineq}
d_{L_p^n(X)}\left((A\otimes I_X^n)(f),\B\left((A\otimes I_X^n)(f)\right)\right)\\\le \beta_p(K)\cdot  d_{L_p^n(X)} \left((A\otimes I_X^n)(f),\B(f)\right).
\end{multline}
\end{lemma}

\begin{proof}
Write
\begin{equation}\label{eq:def a}
a\eqdef d_{L_p^n(X)} \left((A\otimes I_X^n)(f),\B(f)\right),
\end{equation}
and
\begin{equation}\label{eq:def b}
b\eqdef d_{L_p^n(X)} \left(\B(f),\B\left((A\otimes I_X^n)(f)\right)\right).
\end{equation}
Then by the triangle inequality in $L_p^n(X)$,
\begin{equation}\label{eq:a+b}
d_{L_p^n(X)}\left((A\otimes I_X^n)(f),\B\left((A\otimes I_X^n)(f)\right)\right)\le a+b.
\end{equation}
Next, define $\nu\in \P_X$ by
\begin{equation}\label{eq:def nu}
\nu\eqdef \frac{1}{n}\sum_{i=1}^n \d_{(A\otimes I_X^n)(f)(i)}.
\end{equation}
Then
\begin{equation}\label{eq:B mu is}
\B(\nu)=\B\left((A\otimes I_X^n)(f)\right),
\end{equation}
and therefore
\begin{multline}\label{eq:K term compute}
d_{L_p^n(X)}\left((A\otimes I_X^n)(f),\B\left((A\otimes I_X^n)(f)\right)\right)\\=\left(\int_X d_X(y,\B(\nu))^pd\nu(y)\right)^{1/p}.
\end{multline}
It also follows from~\eqref{eq:def a} and~\eqref{eq:def nu} that
\begin{equation}\label{eq:a with mu}
a=\left(\int_X d_X(y,\B(f))^pd\nu(y)\right)^{1/p}.
\end{equation}
By combining~\eqref{eq:def b}, \eqref{eq:B mu is}, \eqref{eq:K term compute} and~\eqref{eq:a with mu}, an application of~\eqref{eq:def p bary} to the measure $\mu=\nu$ with $x=\B(f)$ yields the estimate
\begin{equation}\label{eq:ap on right}
b^p+\frac{1}{K^p}\cdot d_{L_p^n(X)}\left((A\otimes I_X^n)(f),\B\left((A\otimes I_X^n)(f)\right)\right)^p \le a^p.
\end{equation}
The desired estimate~\eqref{eq:beta ineq} now follows by combining~\eqref{eq:a+b} and~\eqref{eq:ap on right} with~\eqref{eq:ab}.
\end{proof}

\begin{lemma}\label{lem:A norm 1 from B}
Fix $p,K\in [1,\infty)$. Suppose that $(X,d_X)$ is a $p$-barycentric
metric space with constant $K$. Then for every $n\in \N$, every
symmetric stochastic matirx $A=(a_{ij})\in M_n(\R)$ and every $f\in
L_p^n(X)$ we have
\begin{multline}\label{eq:contraction}
d_{L_p^n(X)}\left((A\otimes I_X^n)(f),\B(f)\right)\\\le\left(\frac{K^{2p}\gamma_+(A,d_X^p)-1}{K^{2p}\gamma_+(A,d_X^p)+K^p}\right)^{1/p} d_{L_p^n(X)}\left(f,\B(f)\right).
\end{multline}
\end{lemma}
\begin{proof}
For every $i\in \{1,\ldots,n\}$ define $\nu_i\in \P_X$ by
\begin{equation}\label{eq:def mui}
\nu_i\eqdef \sum_{i=1}^na_{ij} \d_{f(j)}.
\end{equation}
Thus
\begin{equation*}
(A\otimes I_X^n)(f)(i)=\B(\nu_i),
\end{equation*}
\begin{equation*}
\int_X d_X(y,\B(\nu_i))^pd\nu_i(y)=\sum_{j=1}^n a_{ij} d_X\left(f(j),(A\otimes I_X^n)(f)(i)\right)^p,
\end{equation*}
and
\begin{equation*}
\int_X d_X(y,\B(f))^pd\nu_i(y)=\sum_{j=1}^n a_{ij} d_X(f(j),\B(f))^p.
\end{equation*}
An application of~\eqref{eq:def p bary} with $\mu=\nu_i$ and $x=\B(f)$ therefore implies that
\begin{multline}\label{eq:to average K}
d_X\left((A\otimes I_X^n)(f)(i),\B(f)\right)^p+\frac{1}{K^p}\sum_{j=1}^n a_{ij} d_X\left(f(j),(A\otimes I_X^n)(f)(i)\right)^p \\\le \sum_{j=1}^n a_{ij} d_X(f(j),\B(f))^p.
\end{multline}
By averaging~\eqref{eq:to average K} over $i\in \{1,\ldots,n\}$ we conclude that
\begin{align}\label{eq:averaged p convex mui}
\nonumber&d_{L_p^n(X)}\left((A\otimes I_X^n)(f),\B(f)\right)^p\\\nonumber&\quad+\frac{1}{nK^p}\sum_{i=1}^n\sum_{j=1}^n a_{ij} d_X\left(f(j),(A\otimes I_X^n)(f)(i)\right)^p
\\&\le \frac{1}{n} \sum_{i=1}^n \sum_{j=1}^n a_{ij} d_X(f(j),\B(f))^p=d_{L_p^n(X)}\left(f,\B(f)\right)^p.
\end{align}

Next, the definition of $\gamma_+(A,d_X^p)$ implies that
\begin{multline}\label{eq:use gamma+}
\frac{1}{n}\sum_{i=1}^n\sum_{j=1}^n a_{ij} d_X\left(f(j),(A\otimes I_X^n)(f)(i)\right)^p\\
\ge
\frac{1}{\gamma_+(A,d_X^p)}\cdot\frac{1}{n^2}\sum_{i=1}^n\sum_{j=1}^n d_X\left(f(j),(A\otimes I_X^n)(f)(i)\right)^p.
\end{multline}
For every $i\in \{1,\ldots, n\}$, an application of~\eqref{eq:def p bary} with $\mu= \frac{1}{n}\sum_{j=1}^n \d_{f(j)}$ and $x=(A\otimes I_X^n)(f)(i)$ implies the estimate
\begin{multline}\label{eq:use p convex again}
d_X\left((A\otimes I_X^n)(f)(i),\B(f)\right)^p+\frac{1}{K^p} \cdot d_{L_p^n(X)} (f,\B(f))^p\\\le \frac{1}{n}\sum_{j=1}^n d_X\left(f(j),(A\otimes I_X^n)(f)(i)\right)^p.
\end{multline}
By averaging~\eqref{eq:use p convex again} over $i\in \{1,\ldots,n\}$ we see that
\begin{multline}\label{eq:averaged second use of p convex}
\frac{1}{n^2}\sum_{i=1}^n\sum_{j=1}^n d_X\left(f(j),(A\otimes I_X^n)(f)(i)\right)^p\\\ge d_{L_p^n(X)}\left((A\otimes I_X^n)(f),\B(f)\right)^p+ \frac{1}{K^p} \cdot d_{L_p^n(X)} (f,\B(f))^p.
\end{multline}
By substituting~\eqref{eq:averaged second use of p convex} into~\eqref{eq:use gamma+}, and plugging the resulting estimate into~\eqref{eq:averaged p convex mui}, we conclude that
\begin{multline*}
\left(1+\frac{1}{K^p\gamma_+(A,d_X^p)}\right)d_{L_p^n(X)}\left((A\otimes I_X^n)(f),\B(f)\right)^p\\
\le \left(1-\frac{1}{K^{2p}\gamma_+(A,d_X^p)}\right)d_{L_p^n(X)} (f,\B(f))^p,
\end{multline*}
which simplifies to give the desired inequality~\eqref{eq:contraction}.
\end{proof}

\begin{proof}[Proof of Lemma~\ref{lem:bound lambda by gamma}]
Simply combine Lemma~\ref{lem:use beta} and Lemma~\ref{lem:A norm 1 from B}.
\end{proof}

\section{Proof of Theorem~\ref{thm:general keith}}\label{sec:ext}

Lemma~\ref{lem:CN} below plays an important role in our proof of
Theorem~\ref{thm:general keith}. It was proved by the second named
author in collaboration with M. Cs\"ornyei (2001); we thank her for
letting us include it here.

\begin{lemma}\label{lem:CN} Fix $m,n\in \N$ and $p\in [1,\infty)$. Let $B=(b_{ij})\in M_{n\times m}(\R)$ and $C=(c_{ij})\in M_n(\R)$ be stochastic matrices (of dimensions $n$ by $m$ and $n$ by $n$, respectively). Fix $\pi \in \Delta^{n-1}$ and suppose that $C$ is reversible relative to $\pi$. Then for every metric space $(X,d_X)$ and every $z_1,\ldots,z_m\in X$ there exist $w_1,\ldots,w_n\in X$ such that
\begin{multline}\label{eq:BC goal}
\max\left\{\sum_{i=1}^n\sum_{r=1}^m \pi_i b_{ir} d_X(w_i,z_r)^p,\sum_{i=1}^n\sum_{j=1}^n \pi_i c_{ij} d_X(w_i,w_j)^p\right\}\\
\le 3^p \sum_{r=1}^m\sum_{s=1}^m (B^*D_\pi CB)_{rs} d_X(z_r,z_s)^p,
\end{multline}
where \begin{equation}\label{eq:def diagonal}
D_\pi=
\begin{pmatrix} \pi_1 & 0 &  \dots&0 \\
  0 & \pi_2& \ddots & \vdots\\
            \vdots & \ddots &  \ddots &0\\
              0 & \dots  &0&\pi_n
                       \end{pmatrix}\in M_n(\R).\end{equation}
\end{lemma}

\begin{proof}
Let $f:\{z_1,\ldots,z_m\}\to \ell_\infty^{m}$ be an isometric embedding (e.g. one can take $f(z)=\sum_{r=1}^{m} d_X(z,z_r)e_r$, where $\{e_r\}_{r=1}^m$ is the standard basis of $\R^m$). Define $y_1,\ldots,y_n\in \ell_\infty^m$ by
\begin{equation}\label{eq:average choice}
\forall\, i\in \{1,\ldots,n\},\quad y_i\eqdef \sum_{r=1}^mb_{ir} f(z_r).
\end{equation}
Next, for every $i\in \{1,\ldots,n\}$ choose $w_i\in \{z_1,\ldots,z_m\}$ such that
\begin{equation}\label{eq:nearest point}
\|y_i-f(w_i)\|_\infty =\min_{z\in \{z_1,\ldots,z_m\}} \|y_i-f(z)\|_\infty.
\end{equation}

By the triangle inequality, for every $i,j\in \{1,\ldots,m\}$ we have
\begin{multline*}
d_X(w_i,w_j)^p=\|f(w_i)-f(w_j)\|_\infty^p\\\le 3^{p-1}\left(\|f(w_i)-y_i\|_\infty^p+\|y_i-y_j\|_\infty^p+\|y_j-f(w_j)\|_\infty^p\right).
\end{multline*}
Consequently, using the stochasticity of $C$ and its reversibility relative to $\pi$,
\begin{multline}\label{eq:two terms C}
\sum_{i=1}^n\sum_{j=1}^n \pi_i c_{ij} d_X(w_i,w_j)^p\\\le 3^{p-1}\sum_{i=1}^n\sum_{j=1}^n \pi_i c_{ij} \|y_i-y_j\|_\infty^p+2\cdot 3^{p-1} \sum_{i=1}^n \pi_i\|y_i-f(w_i)\|_\infty^p.
\end{multline}
Recalling~\eqref{eq:average choice}, the convexity of the function $v\mapsto \|v\|_\infty^p$ implies that
\begin{multline}\label{eq:get matrix product}
\sum_{i=1}^n\sum_{j=1}^n \pi_i c_{ij} \|y_i-y_j\|_\infty^p\le \sum_{i=1}^n\sum_{j=1}^n \pi_ic_{ij} \sum_{r=1}^m\sum_{s=1}^m b_{ir} b_{js}\|f(z_r)-f(z_s)\|_\infty^p\\
=\sum_{r=1}^m\sum_{s=1}^m (B^*D_\pi CB)_{rs} d_X(z_r,z_s)^p.
\end{multline}
Next, due to~\eqref{eq:nearest point} and the fact that $CB$ is a stochastic matrix,
$$
\forall\, i\in \{1,\ldots,n\},\quad \|y_i-f(w_i)\|_\infty^p\le \sum_{r=1}^m (CB)_{ir} \|y_i-f(z_r)\|_\infty^p.
$$
Recalling~\eqref{eq:average choice} and using the convexity of the function $v\mapsto \|v\|_\infty^p$ once more, we deduce that
\begin{multline}\label{eq:get matrix product again}
\sum_{i=1}^n \pi_i\|y_i-f(w_i)\|_\infty^p\le \sum_{i=1}^n\pi_i\sum_{r=1}^m (CB)_{ir}\sum_{s=1}^m b_{is}\|f(z_s)-f(z_r)\|_\infty^p\\
=\sum_{r=1}^m\sum_{s=1}^m (B^*D_\pi CB)_{rs}d_X(z_r,z_s)^p.
\end{multline}
A combination of~\eqref{eq:two terms C}, \eqref{eq:get matrix product} and~\eqref{eq:get matrix product again} now implies that
\begin{equation}\label{eq:C bound}
\sum_{i=1}^n\sum_{j=1}^n \pi_ic_{ij} d_X(w_i,w_j)^p\le 3^p\sum_{r=1}^m\sum_{s=1}^m (B^*D_\pi CB)_{rs} d_X(z_r,z_s)^p.
\end{equation}

Next, note that by the triangle inequality for every $i,j\in \{1,\ldots,n\}$ and $r\in \{1,\ldots,m\}$ we have
\begin{multline}\label{eq:put yj in}
d_X(w_i,z_r)^p=\|f(w_i)-f(z_r)\|_\infty^p\\
\le 3^{p-1}\left(\|f(w_i)-y_i\|_\infty^p+\|y_i-y_j\|_\infty^p+\|y_j-f(z_r)\|_\infty^p\right).
\end{multline}
By multiplying inequality~\eqref{eq:put yj in} by $\pi_i b_{ir}c_{ij}$, summing over $r\in \{1,\ldots,m\}$ and $i,j\in \{1,\ldots,n\}$, and using the stochasticity of $B$ and $C$, we deduce that
\begin{multline}\label{eq:three terms B}
\frac{1}{3^{p-1}}\sum_{i=1}^n\sum_{r=1}^m \pi_i b_{ir} d_X(w_i,z_r)^p
\le  \sum_{i=1}^n\pi_i\|f(w_i)-y_i\|_\infty^p\\+\sum_{i=1}^n\sum_{j=1}^n \pi_ic_{ij}\|y_i-y_j\|_\infty^p+\sum_{j=1}^n\sum_{r=1}^m (B^*D_\pi C)_{rj}\|y_j-f(z_r)\|_\infty^p.
\end{multline}
Recalling~\eqref{eq:average choice} and using the convexity of the function $v\mapsto \|v\|_\infty^p$, we have
\begin{multline}\label{eq:get matrix product yet again}
\sum_{j=1}^n\sum_{r=1}^m (B^*D_\pi C)_{rj}\|y_j-f(z_r)\|_\infty^p\\\le
\sum_{r=1}^m\sum_{s=1}^m (B^*D_\pi CB)_{rs}d_X(z_r,z_s)^p.
\end{multline}
A combination of~\eqref{eq:three terms B} with~\eqref{eq:get matrix product}, \eqref{eq:get matrix product again}, and~\eqref{eq:get matrix product yet again} yields the estimate
\begin{equation*}\label{eq:B bound}
\sum_{i=1}^n\sum_{r=1}^m \pi_ib_{ir} d_X(w_i,z_r)^p\le 3^p \sum_{r=1}^m\sum_{s=1}^m (B^*D_\pi CB)_{rs} d_X(z_r,z_s)^p,
\end{equation*}
which, due to~\eqref{eq:C bound}, yields the desired estimate~\eqref{eq:BC goal}.
\end{proof}

The following lemma is  a natural variant of~\cite[Lem.~1.1]{Ball}.

\begin{lemma}[dual extension criterion]\label{lem:ball 1.1}
Fix $p,\Lambda,\Gamma\in [1,\infty)$, an integer $n$, and $\e\in (0,1)$. Suppose
that $(X,d_X)$ and $(Y,d_Y)$ are metric spaces with $(Y,d_Y)$ being
$W_p$ barycentric with constant $\Gamma$. Fix  $Z\subseteq X$, a
Lipschitz function $f:Z\to Y$ and distinct $x_1,\ldots,x_n\in X$. Suppose
that for every symmetric $n$ by $n$ matrix $H=(h_{ij})$ with
nonnegative entries there exists $\Phi^H:\{x_1,\ldots,x_n\}\to Y$
with $\Phi^H|_{\{x_1,\ldots,x_n\}\cap Z}=f|_{\{x_1,\ldots,x_n\}\cap
Z}$ such that
    \begin{equation*}
\sum_{i=1}^n\sum_{j=1}^n h_{ij} d_Y\left(\Phi^H(x_i),\Phi^H(x_j)\right)^p\le \Lambda^p\|f\|_{\Lip}^p \sum_{i=1}^n\sum_{j=1}^n h_{ij} d_X(x_i,x_j)^p.
    \end{equation*}
Then there is $F:\{x_1,\ldots,x_n\}\to Y$ with $F|_{\{x_1,\ldots,x_n\}\cap Z}=f|_{\{x_1,\ldots,x_n\}\cap Z}$ and
$\left\|F\right\|_{\Lip}\le (1+\e)\Gamma \Lambda\|f\|_{\Lip}$.
\end{lemma}

\begin{proof}
We proceed via the following duality argument due to Ball~\cite{Ball} (which is itself inspired by the work of Maurey~\cite{Mau74}), with a slight twist that brings in the assumption that $Y$ is $W_p$ barycentric.

Consider the following set of $n$ by $n$ symmetric matrices.
\begin{multline*}
C\eqdef \Big\{\left(d_Y(\Phi(x_i),\Phi(x_j))^p\right)\in M_n(\R):\\ \Phi:\{x_1,\ldots,x_n\}\to Y\ \wedge\ \Phi|_{\{x_1,\ldots,x_n\}\cap Z}=f|_{\{x_1,\ldots,x_n\}\cap Z}\Big\}.
\end{multline*}
Let $D\subseteq M_n(\R)$ be the set of all $n$ by $n$ symmetric
matrices with nonnegative entries and define
$$
E\eqdef \overline{\conv\left(C+D\right)}.
$$
For every $i,j\in \{1,\ldots, n\}$ write
$t_{ij}\eqdef\Lambda^p\|f\|_{\Lip}^pd_X(x_i,x_j)^p$. The assumption
of Lemma~\ref{lem:ball 1.1} can be rephrased as
\begin{equation}\label{eq:for duality}
\sup_{H\in D}\inf_{M\in C}\sum_{i=1}^n\sum_{j=1}^n h_{ij}\left(m_{ij}-t_{ij}\right)\le 0.
\end{equation}
It follows that the matrix $T=(t_{ij})$ belongs to $E$, since otherwise by the separation theorem (Hahn-Banach) there would exist a symmetric matrix $H=(h_{ij})$  such that
\begin{equation}\label{eq;alpha lower}
\inf_{M\in E} \sum_{i=1}^n\sum_{j=1}^n h_{ij}m_{ij} >\sum_{i=1}^n\sum_{j=1}^n h_{ij}t_{ij}
\end{equation}
Since $E\supseteq C+D$ it follows from~\eqref{eq;alpha lower} that the entries of $H$ are nonnegative, i.e., $H\in D$.  Now~\eqref{eq;alpha lower} contradicts~\eqref{eq:for duality} since $E\supseteq C$.

Having shown that $T\in E$, we deduce that there exists  $m\in \N$ and $\lambda_1,\ldots,\lambda_m\in [0,1]$ with $\sum_{k=1}^m \lambda_k=1$, and in addition there are $\Phi^1,\ldots,\Phi^m:\{x_1,\ldots,x_n\}\to Y$ with $\Phi^k|_{\{x_1,\ldots,x_n\}\cap Z}=f|_{\{x_1,\ldots,x_n\}\cap Z}$ for all $k\in \{1,\ldots,m\}$, such that for every $i,j\in \{1,\ldots, m\}$,
\begin{equation}\label{eq: inE}
(1+\e)^p \Lambda^p\|f\|_{\Lip}^pd_X(x_i,x_j)^p- \sum_{k=1}^m \lambda_k d_Y\left(\Phi^k(x_i),\Phi^k(x_j)\right)^p\ge 0.
\end{equation}

For every $i\in \{1,\ldots,n\}$ consider the finitely supported probability measure $\mu_i$ on $Y$ given by
$$
\mu_i\eqdef \sum_{k=1}^m \lambda_k\delta_{\Phi^k(x_i)},
$$
and set
$$
F(x_i)\eqdef \B\left(\mu_i\right).
$$
If $x_i\in Z$ then $\mu_i=\delta_{f(x_i)}$ and therefore $F(x_i)=f(x_i)$. In other words, $F|_{\{x_1,\ldots,x_n\}\cap Z}=f|_{\{x_1,\ldots,x_n\}\cap Z}$. Also, for every $i,j\in \{1,\ldots,n\}$ we have
\begin{eqnarray*}
d_Y(F(x_i),F(x_j))&\stackrel{\eqref{eq:Gamma Wp Lip}}{\le}& \Gamma
W_p(\mu_i,\mu_j)\\&\le& \Gamma\left(\sum_{k=1}^m\lambda_k
d_Y(\Phi^k(x_i),\Phi^k(x_j))^p\right)^{1/p}\\&\stackrel{\eqref{eq:
inE}}{\le}& (1+\e)\Gamma \Lambda\|f\|_{\Lip} d_Y(x_i,x_j).
\end{eqnarray*}
Thus $\|F\|_{\Lip}\le (1+\e)\Gamma \Lambda\|f\|_{\Lip}$, as required.
\end{proof}

\begin{proof}[Proof of Theorem~\ref{thm:general keith}]
Fix $m,n\in \N$. Take $x_1,\ldots,x_n\in X\setminus Z$ and $z_1,\ldots,z_m\in Z$. If $H=(h_{ij})\in M_{n+m}(\R)$ is symmetric then write
\[ H =\begin{pmatrix} U(H) & W(H)^*\\ W(H)& V(H) \end{pmatrix}, \]
where $U(H)\in M_m(\R)$, $V(H)\in M_n(\R)$ and $W(H)\in M_{n\times m}(\R)$. With this notation, define
\begin{multline*}
R_H\eqdef \sum_{r=1}^m\sum_{s=1}^m U(H)_{rs} d_X(z_r,z_s)^p\\
+2\sum_{i=1}^n\sum_{r=1}^m W(H)_{ir}d_X(x_i,z_r)^p+\sum_{i=1}^n\sum_{j=1}^n V(H)_{ij}d_X(x_i,x_j)^p,
\end{multline*}
and  for every $y_1,\ldots,y_n\in Y$,
\begin{multline*}
L_H(y_1,\ldots,y_n)\eqdef \sum_{r=1}^m\sum_{s=1}^m U(H)_{rs} d_Y(f(z_r),f(z_s))^p\\
+2\sum_{i=1}^n\sum_{r=1}^m W(H)_{ir}d_Y(y_i,f(z_r))^p+\sum_{i=1}^n\sum_{j=1}^n V(H)_{ij}d_Y(y_i,y_j)^p.
\end{multline*}

Fix from now on $M\in (M_p(X),2M_p(X)]$ and $N\in (N_p(Y),2N_p(Y)]$. Due to Lemma~\ref{lem:ball 1.1} it suffices to show that for every symmetric matrix $H=(h_{ij})\in M_{n+m}(\R)$ with nonnegative entries and for every $\d\in (0,1)$ one can find $y_1,\ldots,y_n\in Y$ such that
\begin{equation*}\label{eq:goal L less than R}
L_H(y_1,\ldots,y_n)\le \Lambda( R_H+\d),
\end{equation*}
where
\begin{equation}\label{eq:defK}
\Lambda\eqdef \frac{18^p}{3}(N^p+1)M^p\|f\|_{\Lip}^p.
\end{equation}
Since the Lipschitz condition of $f$ on $\{z_1,\ldots,z_m\}$ implies that
$$
\sum_{r=1}^m\sum_{s=1}^m U(H)_{rs} d_Y(f(z_r),f(z_s))^p\le \|f\|_{\Lip}^p \sum_{r=1}^m\sum_{s=1}^m U(H)_{rs} d_X(z_r,z_s)^p,
$$
it suffices to establish the existence of $y_1,\ldots,y_n\in Y$ that satisfy the inequality
\begin{multline}\label{eq:goal without upper corner}
\frac{1}{\Lambda}\left(2\sum_{i=1}^n\sum_{r=1}^m W(H)_{ir}d_Y(y_i,f(z_r))^p+\sum_{i=1}^n\sum_{j=1}^n V(H)_{ij}d_Y(y_i,y_j)^p\right)\\
\le 2\sum_{i=1}^n\sum_{r=1}^m W(H)_{ir}d_X(x_i,z_r)^p+\sum_{i=1}^n\sum_{j=1}^n V(H)_{ij}d_X(x_i,x_j)^p+\d.
\end{multline}

Fix $t\in \N$ and $\e\in (0,1)$ that will be determined later.
 Note that the diagonal entries $\{V(H)_{ii}\}_{i=1}^n$ are irrelevant for the validity of~\eqref{eq:goal without upper corner}, so we may assume from now on that $V(H)_{ii}=0$ for all $i\in \{1,\ldots,n\}$.

Define $\pi \in \R^n$ by
\begin{equation}\label{eq:choose pi}
\forall\, i\in \{1,\ldots,n\},\quad \pi_i\eqdef \frac{m\e+\sum_{r=1}^mW(H)_{ir}}{mn\e+\sum_{j=1}^n\sum_{r=1}^m W(H)_{jr}}.
\end{equation}
Thus $\pi\in \Delta^{n-1}$. Next, define $B=(b_{ir})\in M_{n\times m}(\R)$ by setting for every $i\in \{1,\ldots,n\}$ and $r\in \{1,\ldots,m\}$,
\begin{equation}\label{eq:choose B}
b_{ir}\eqdef \frac{\e+W(H)_{ir}}{m\e+\sum_{s=1}^mW(H)_{is}}.
\end{equation}
Thus $B$ is a stochastic matrix.  Finally, define $A=(a_{ij})\in M_n(\R)$ by setting for $i\in \{1,\ldots,n\}$,
\begin{equation}\label{eq:choose diaginal A}
a_{ii}\eqdef 1-\frac{2^{p+1}}{(2^{p}+1)M^p t}\cdot\frac{\sum_{j=1}^n V(H)_{ij}}{m\e+\sum_{r=1}^m W(H)_{ir}},
\end{equation}
and for distinct $i,j\in \{1,\ldots,n\}$,
\begin{equation}\label{eq:choose A off diagonal}
a_{ij}\eqdef \frac{2^{p+1}}{(2^{p}+1)M^p t}\cdot \frac{V(H)_{ij}}{m\e+\sum_{r=1}^m W(H)_{ir}}.
\end{equation}

The role of $\e$ is only to ensure that the denominators that appear in~\eqref{eq:choose pi}, \eqref{eq:choose B}, \eqref{eq:choose diaginal A} and~\eqref{eq:choose A off diagonal} do not vanish. An inspection of the ensuing argument reveals that there is flexibility in the choice of the normalizing factors in~\eqref{eq:choose diaginal A} and~\eqref{eq:choose A off diagonal}; the choices above were made in order to simplify some expressions in what follows. Fixing $\e$, we will assume from now on that $t$ is sufficiently large so as to ensure that $a_{11},\ldots,a_{nn}$ are all nonnegative. Thus $A$ is a stochastic matrix. Note also that since $V(H)$ is symmetric, an inspection of~\eqref{eq:choose pi} and~\eqref{eq:choose A off diagonal} reveals that $A$ is reversible relative to $\pi$.

Set
\begin{equation}\label{eq:choose tau}
\tau\eqdef \left\lceil \frac{t}{2^p}\right\rceil.
\end{equation}
By Lemma~\ref{lem:CN} applied with $C=\A_{\tau}(A)$ there exist $w_1,\ldots,w_n\in Y$ such that
\begin{multline}\label{eq:use CN}
\max\left\{\sum_{i=1}^n\sum_{r=1}^m \pi_ib_{ir}d_Y(w_i,f(z_r))^p,\sum_{i=1}^n\sum_{j=1}^n \pi_i \A_{\tau}(A)_{ij}d_Y(w_i,w_j)^p\right\}\\
\le 3^p\sum_{r=1}^m\sum_{s=1}^m \left(B^*D_\pi \A_{\tau}(A)B\right)_{rs} d_Y\left(f(z_r),f(z_s)\right)^p.
\end{multline}
Since $N>N_p(Y)$ there exist $y_1,\ldots,y_n\in Y$ such that
\begin{multline}\label{eq:N appears}
\sum_{i=1}^n \pi_i d_Y(w_i,y_i)^p+\frac{t}{2^p}\sum_{i=1}^n\sum_{j=1}^n \pi_i  a_{ij} d_Y(y_i,y_j)^p\\
\le N^p\sum_{i=1}^n\sum_{j=1}^n \pi_i\A_\tau(A)_{ij} d_Y(w_i,w_j)^p.
\end{multline}
We will now show that the points $y_1,\ldots, y_n$ thus found satisfy the desired inequality~\eqref{eq:goal without upper corner}.

To estimate the left hand side of~\eqref{eq:goal without upper corner} from above, denote
\begin{equation}\label{eq:def normalization theta}
\theta\eqdef \frac{(2^{p}+1)M^p(t+1)}{2^{p+1}}\left(mn\e+\sum_{i=1}^n\sum_{r=1}^m W(H)_{ir}\right),
\end{equation}and observe that due to~\eqref{eq:choose pi}, \eqref{eq:choose B} and~\eqref{eq:def normalization theta} we have
\begin{multline}\label{eq:W identity}
\forall(i,r)\in \{1,\ldots,n\}\times \{1,\ldots, m\}, \\ \frac{(2^{p}+1)M^p}{2^{p+1}}\left(W(H)_{ir}+\e\right)= \frac{\theta}{t+1}\pi_ib_{ir}.
\end{multline}
Similarly, due to~\eqref{eq:choose pi}, \eqref{eq:choose diaginal A} and \eqref{eq:choose A off diagonal}  we have
\begin{equation}\label{eq:V identity}
 \left(i,j\in \{1,\ldots, n\}\ \wedge\  i\neq j\right)\implies  V(H)_{ij}= \frac{\theta  t}{t+1}\pi_i a_{ij}.
\end{equation}
Hence,
\begin{multline}\label{eq:plug upper estimate matrices}
\frac{t+1}{\theta}\left(2\sum_{i=1}^n\sum_{r=1}^m W(H)_{ir}d_Y(y_i,f(z_r))^p+\sum_{i=1}^n\sum_{j=1}^n V(H)_{ij}d_Y(y_i,y_j)^p\right)\\
\le 2\sum_{i=1}^n\sum_{r=1}^m \pi_ib_{ir}d_Y(y_i,f(z_r))^p+t\sum_{i=1}^n\sum_{j=1}^n\pi_i a_{ij}d_Y(y_i,y_j)^p.
\end{multline}

By the triangle inequality, for every $i\in \{1,\ldots,n\}$ and $r\in \{1,\ldots,m\}$ we have
$$
d_Y(y_i,f(z_r))^p\le 2^{p-1} d_Y(y_i,w_i)^p+2^{p-1} d_Y(w_i,f(z_r))^p.
$$
Consequently,
\begin{multline}\label{eq;p-1}
\sum_{i=1}^n\sum_{r=1}^m \pi_ib_{ir}d_Y(y_i,f(z_r))^p\\
\le 2^{p-1}\sum_{i=1}^n \pi_i d_Y(y_i,w_i)^p+2^{p-1} \sum_{i=1}^n\sum_{r=1}^m \pi_ib_{ir}d_Y(w_i,f(z_r))^p.
\end{multline}
We can therefore bound the right hand side of~\eqref{eq:plug upper estimate matrices} as follows.
\begin{align}\label{eq:set up for lip condition}
&\nonumber2\sum_{i=1}^n\sum_{r=1}^m \pi_ib_{ir}d_Y(y_i,f(z_r))^p+t\sum_{i=1}^n\sum_{j=1}^n\pi_i a_{ij}d_Y(y_i,y_j)^p\\
&\nonumber\stackrel{\eqref{eq;p-1}}{\le}2^p \sum_{i=1}^n \pi_i d_Y(y_i,w_i)^p+t\sum_{i=1}^n\sum_{j=1}^n\pi_i a_{ij}d_Y(y_i,y_j)^p
\\&\nonumber\quad +2^p \sum_{i=1}^n\sum_{r=1}^m \pi_ib_{ir}d_Y(w_i,f(z_r))^p\\
&\nonumber \stackrel{\eqref{eq:N appears}}{\le} (2N)^p\sum_{i=1}^n\sum_{j=1}^n \pi_i\A_\tau(A)_{ij} d_Y(w_i,w_j)^p+2^p \sum_{i=1}^n\sum_{r=1}^m \pi_ib_{ir}d_Y(w_i,f(z_r))^p\\
& \nonumber \stackrel{\eqref{eq:use CN}}{\le} 6^p\left(N^p+1\right)\sum_{r=1}^m\sum_{s=1}^m \left(B^*D_\pi \A_{\tau}(A)B\right)_{rs} d_Y(f(z_r),f(z_s))^p\\
&\le 6^p\left(N^p+1\right)\|f\|_{\Lip}^p \sum_{r=1}^m\sum_{s=1}^m \left(B^*D_\pi \A_{\tau}(A)B\right)_{rs}d_X(z_r,z_s)^p.
\end{align}

For every $i,j\in \{1,\ldots,n\}$ and $r,s\in \{1,\ldots,m\}$ we have
\begin{equation*}\label{eq:final triangle in proof}
d_X(z_r,z_s)^p\le 3^{p-1} d_X(z_r,x_i)^p+3^{p-1} d_X(x_i,x_j)^p+3^{p-1}d_X(x_j,z_s)^p.
\end{equation*}
Consequently,
\begin{align}\label{eq:S13}
&\nonumber\sum_{r=1}^m\sum_{s=1}^m \left(B^*D_\pi \A_{\tau}(A)B\right)_{rs}d_X(z_r,z_s)^p\\\nonumber&=\sum_{r=1}^m\sum_{s=1}^m \sum_{i=1}^n\sum_{j=1}^n b_{ir}b_{js}\pi_i\A_\tau(A)_{ij}d_X(z_r,z_s)^p\\&\le 3^{p-1}(S_1+S_2+S_3),
\end{align}
where, using the stochasticity of $A$ and $B$,
\begin{eqnarray}\label{eq:def S1}
S_1&\eqdef&\nonumber \sum_{r=1}^m\sum_{s=1}^m \sum_{i=1}^n\sum_{j=1}^n b_{ir}b_{js}\pi_i\A_\tau(A)_{ij} d_X(z_r,x_i)^p\\&=& \sum_{i=1}^n\sum_{r=1}^m \pi_ib_{ir} d_X(x_i,z_r)^p,
\end{eqnarray}
using the stochasticity of $A$ and $B$ and the fact that $X$ has Markov type $p$ with $M>M_p(X)$,
\begin{eqnarray}\label{eq:def S2}
S_2\nonumber&\eqdef& \sum_{r=1}^m\sum_{s=1}^m \sum_{i=1}^n\sum_{j=1}^n b_{ir}b_{js}\pi_i\A_\tau(A)_{ij} d_X(x_i,x_j)^p\\\nonumber&=&\frac{1}{\tau}\sum_{k=0}^{\tau-1} \sum_{i=1}^n\sum_{j=1}^n\pi_i(A^k)_{ij} d_X(x_i,x_j)^p
\\&\stackrel{\eqref{eq:to reverse type}}{\le}& \nonumber
\frac{M^p}{\tau}\sum_{k=1}^{\tau} k\sum_{i=1}^n\sum_{j=1}^n \pi_i a_{ij} d_X(x_i,x_j)^p\\
\nonumber&=& \frac{M^p(\tau+1)}{2}\sum_{i=1}^n\sum_{j=1}^n \pi_i a_{ij} d_X(x_i,x_j)^p\\
&\stackrel{\eqref{eq:choose tau}}{\le}& \frac{(2^{p}+1)M^pt}{2^{p+1}}\sum_{i=1}^n\sum_{j=1}^n \pi_i a_{ij} d_X(x_i,x_j)^p,
\end{eqnarray}
and, using the stochasticity of $A$ and $B$, and the reversibility of $\A_\tau(A)$ relative to $\pi$,
\begin{eqnarray}\label{eq:def S3}
S_3&\eqdef&\nonumber \sum_{r=1}^m\sum_{s=1}^m \sum_{i=1}^n\sum_{j=1}^n b_{ir}b_{js}\pi_i\A_\tau(A)_{ij} d_X(x_j,z_s)^p\\
&=&\nonumber\sum_{i=1}^n\sum_{j=1}^n\sum_{s=1}^m b_{js}\pi_j\A_\tau(A)_{ji}d_X(x_j,z_s)^p\\&=& \sum_{j=1}^n\sum_{s=1}^m \pi_j b_{js}d_X(x_j,z_s)^p=S_1.
\end{eqnarray}

By substituting~\eqref{eq:def S1}, \eqref{eq:def S2} and~\eqref{eq:def S3} into~\eqref{eq:S13}, and combining the resulting estimate with~\eqref{eq:set up for lip condition} and~\eqref{eq:plug upper estimate matrices} while recalling the definition of $\Lambda$ in~\eqref{eq:defK}, we arrive at
\begin{align*}
&\frac{1}{\Lambda}\left(2\sum_{i=1}^n\sum_{r=1}^m W(H)_{ir}d_Y(y_i,f(z_r))^p+\sum_{i=1}^n\sum_{j=1}^n V(H)_{ij}d_Y(y_i,y_j)^p\right)\\
&\le  2\sum_{i=1}^n\sum_{r=1}^m \frac{2^{p+1}\theta\pi_ib_{ir}}{(2^p+1)M^p(t+1)} d_X(x_i,z_r)^p+ \sum_{i=1}^n\sum_{j=1}^n \frac{t\theta\pi_i a_{ij}}{t+1} d_X(x_i,x_j)^p\\
&= 2\sum_{i=1}^n\sum_{r=1}^m \left(W(H)_{ir}+\e\right) d_X(x_i,z_r)^p+ \sum_{i=1}^n\sum_{j=1}^n V(H)_{ij}d_X(x_i,x_j)^p,
\end{align*}
where we used the identities~\eqref{eq:W identity} and~\eqref{eq:V identity}. Choosing $\e\in (0,1)$ so that $2\e\sum_{i=1}^n\sum_{r=1}^m d_X(x_i,z_r)^p\le \d$ yields the desired estimate~\eqref{eq:goal without upper corner}.
\end{proof}

\section{Proof of Theorem~\ref{thm:no cotype} and Theorem~\ref{thm:aspect}}\label{sec:kalton}

Both Theorem~\ref{thm:no cotype} and Theorem~\ref{thm:aspect} rely on a quantitative variant of a beautiful construction of Kalton~\cite{Kal04,Kal12}. Before passing to the construction itself, we record some basic facts on metric Markov cotype.

\begin{lemma}\label{lem:retraction cotype}
Fix $p\in (0,\infty)$ and let $(X,d_X)$ be a metric space with metric Markov cotype $p$. Suppose that $S\subseteq X$ is a Lipschitz retract of $X$, i.e., there exists a Lipschitz mapping $\rho:X\to S$ such that $\rho(s)=s$ for every $s\in S$. Then $S$ also has metric Markov cotype $p$, and in fact
$$
N_p(S)\le \|\rho\|_{\Lip}N_p(X).
$$
\end{lemma}

\begin{proof}
Fix $N>N_p(X)$ and $n,t\in \N$. Suppose that $A=(a_{ij})\in M_n(\R)$
is a stochastic matrix that is reversible relative to $\pi\in
\Delta^{n-1}$. For every $x_1,\ldots,x_n\in S$ there exist
$y_1,\ldots,y_n\in X$ such that
\begin{multline*}
\sum_{i=1}^n \pi_id_X(x_i,y_i)^p+t\sum_{i=1}^n\sum_{j=1}^n \pi_ia_{ij}
d_X(y_i,y_j)^p\\\le N^p\sum_{i=1}^n\sum_{j=1}^n
\pi_i\A_t(A)_{ij}d_X(x_i,x_j)^p.
\end{multline*}
Then, since $\rho(x_i)=x_i$ for every $i\in \{1,\ldots,n\}$,
\begin{align*}
&\sum_{i=1}^n \pi_id_X(x_i,\rho(y_i))^p+t\sum_{i=1}^n\sum_{j=1}^n \pi_ia_{ij}
d_X(\rho(y_i),\rho(y_j))^p\\
&\le\|\rho\|_{\Lip}^p\left (\sum_{i=1}^n \pi_id_X(x_i,y_i)^p+t\sum_{i=1}^n\sum_{j=1}^n \pi_ia_{ij}
d_X(y_i,y_j)^p \right )\\
& \le \|\rho\|^p_{\Lip}N^p\sum_{i=1}^n\sum_{j=1}^n
\pi_i\A_t(A)_{ij}d_X(x_i,x_j)^p.\qedhere
\end{align*}
\end{proof}

We next show that the real line $\R$ (equipped with the standard
metric) fails to have metric Markov cotype $p$ for any $p\in (0,2)$.
Lemma~\ref{lem:no cotype R} below contains a simple explicit example
that exhibits this fact, but when $p\in [1,2)$ there is a also a
roundabout way to see that $\R$ fails to have metric Markov cotype
$p$ via the link to Lipschitz extension. Indeed, it follows from the
definition of metric Markov cotype that $N_p(\ell_p)=N_p(\R)$. Ball
proved~\cite{Ball} that the Markov type $p$ constant of $\ell_p$
satisfies $M_p(\ell_p)=1$ for $p\in [1,2)$. Corollary~\ref{coro:bona
fide lip ext} therefore implies that $e(\ell_p,\ell_p)\lesssim
N_p(\R)$. But it is known that $e(\ell_p,\ell_p)=\infty$ for every
$p\in [1,2)$: for $p\in (1,2)$ see~\cite{Nao01}, and for $p=1$ it is
observed in~\cite{MM10} that this follows from~\cite{Bou81}
or~\cite{FJS88} (\cite{MM10} also provides an interesting third
proof of the fact that $e(\ell_1,\ell_1)=\infty$).

\begin{lemma}\label{lem:no cotype R}
For every $p\in (0,2)$ the real line $\R$ (equipped with the standard metric) fails to have metric Markov cotype $p$.
\end{lemma}

\begin{proof} Fix $p\in (0,2)$ and suppose for the sake of obtaining a contradiction that $N_p(\R)<\infty$. Fixing $n\in\N$, define $A=(a_{ij})\in M_n(\R)$ by $a_{11}=a_{nn}=1/2$ and $a_{i,i+1}=a_{i+1,i}=1/2$ for every $i\in \{1,\ldots,n-1\}$, and the remaining entries of $A$ vanish. Thus $A$ is a symmetric stochastic matrix, corresponding to the standard random walk on $\{1,\ldots,n\}$ in which if the walker is at either $1$ or $n$ then with probability $1/2$  it does nothing in the next step, and with probability $1/2$ it moves in the next step to its unique neighbor in $\{2,n-1\}$. Let $\{W_0,W_1,W_2,\ldots\}$ denote this walk, i.e., $W_0$ is uniformly distributed on $\{1,\ldots,n\}$ and conditioned on $W_t=i$ the probability that $W_{t+1}=j$ equals $a_{ij}$. Thus, for every $t\in \N$ we have
\begin{multline}\label{eq:drunk}
\left(\frac{1}{n}\sum_{i=1}^n\sum_{j=1}^n (A^t)_{ij}|i-j|^p\right)^{1/p}=\left(\E\left[|W_t-W_0|^p\right]\right)^{1/p}\\\le \left(\E\left[|W_t-W_0|^2\right]\right)^{1/2}\le \sqrt{t}.
\end{multline}
To justify the final inequality in~\eqref{eq:drunk}, proceed by induction on $t$ as follows. Since $|W_{t+1}-W_t|\le 1$ point-wise,
$$
\E\left[|W_{t+1}-W_0|^2\right]\le \E\left[|W_t-W_0|^2\right]+2\E\left[(W_t-W_0)(W_{t+1}-W_t)\right]+1,
$$
so for the induction step  it suffices to show that for every $t\in \N$ we have
$\E\left[(W_t-W_0)(W_{t+1}-W_t)\right]\le 0$.  By conditioning on $W_0,W_t$, it suffices to check the point-wise inequality
\begin{equation}\label{eq:conditioned}
(W_t-W_0)\E\left[W_{t+1}-W_t\big|W_0,W_t\right]\le 0.
\end{equation}
\eqref{eq:conditioned} is easy to verify: if $W_t\in \{2,\ldots,n-1\}$ then $W_{t+1}-W_t$ is uniformly distributed on $\{-1,1\}$ and therefore $\E\left[W_{t+1}-W_t\big|W_0,W_t\right]=0$, if $W_t=1$ then $W_t-W_0\le 0$ and $W_{t+1}-W_t\ge 0$, and if  $W_t=n$ then $W_t-W_0\ge 0$ and $W_{t+1}-W_t\le 0$.

Due to~\eqref{eq:drunk}, for every $t\in \N$ we have
$$
\frac{1}{n}\sum_{i=1}^n\sum_{j=1}^n \A_t(A)_{ij}|i-j|^p\le \frac{1}{t}\sum_{s=1}^t s^{p/2}\le t^{p/2}.
$$
The definition of metric Markov cotype $p$ therefore implies that there exist $y_1,\ldots,y_n\in \R$ such that
\begin{equation}\label{eq:contra assumption cotype R}
\frac{1}{n}\sum_{i=1}^n |i-y_i|^p+\frac{t}{n}\sum_{i=1}^{n-1}|y_{i+1}-y_i|^p\le N_p(\R)^p t^{p/2}.
\end{equation}

Suppose first that $p\in [1,2)$. In this case we choose
\begin{equation}\label{eq:first choice t}
t=\left\lceil\left(4N_p(\R)\right)^{\frac{2p}{2-p}}\right\rceil.
\end{equation}
Using H\"older's inequality and~\eqref{eq:contra assumption cotype R} we deduce that
$$
\sum_{i=1}^{n-1}|y_{i+1}-y_i|\le (n-1)^{1-\frac{1}{p}}\left(\sum_{i=1}^{n-1}|y_{i+1}-y_i|^{p}\right)^{1/p}\le
\frac{N_p(\R)n}{t^{\frac{1}{p}-\frac{1}{2}}}.
$$
Consequently, for every $i\in \{2,\ldots,n\}$ we have
\begin{equation}\label{eq:yi close to y1}
|y_i-y_1|\le \sum_{j=1}^{i-1}|y_{j+1}-y_j|\le \frac{N_p(\R)n}{t^{\frac{1}{p}-\frac{1}{2}}}\stackrel{\eqref{eq:first choice t}}{\le} \frac{n}{4}.
\end{equation}
\eqref{eq:yi close to y1} implies that if $y_1\le n/2$ then $y_i\le 3n/4$ for every $i\in \{1,\ldots,n\}$ and if $y_1\ge n/2$ then $y_i\ge n/4$ for every $i\in \{1,\ldots,n\}$. Hence,
\begin{equation}\label{eq:first term in cotype contra big}
\frac{1}{n}\sum_{i=1}^n |i-y_i|^p\ge \frac{1}{n}\sum_{i=1}^{\lfloor n/8\rfloor } |i-y_i|^p +\frac{1}{n}\sum_{i=\lceil 7n/8\rceil}^n |i-y_i|^p\gtrsim n^{p}.
\end{equation}
By substituting~\eqref{eq:first term in cotype contra big} into~\eqref{eq:contra assumption cotype R} and recalling~\eqref{eq:first choice t} we conclude that
$$
n\le 4^{\frac{p}{2-p}} N_p(\R)^{\frac{2}{2-p}},
$$
which is a contradiction for large enough $n$.

It remains to deal with the case $p\in (0,1]$. Now our choice of $t$ is
\begin{equation}\label{eq:second choice t}
t=\left\lceil\left(4N_p(\R)\right)^{\frac{2p}{2-p}}n^{\frac{2(1-p)}{2-p}}\right\rceil.
\end{equation}
Observe that since $p\in (0,1]$, for every $i\in \{2,\ldots,n\}$,
\begin{multline*}
|y_i-y_1|\le \sum_{j=1}^{i-1}|y_{j+1}-y_j|\\\le \left(\sum_{j=1}^{i-1}|y_{j+1}-y_j|^p\right)^{1/p}\stackrel{\eqref{eq:contra assumption cotype R} }{\le} \frac{N_p(\R)n^{\frac{1}{p}}}{t^{\frac{1}{p}-\frac{1}{2}}}\stackrel{\eqref{eq:second choice t}}{\le}\frac{n}{4}.
\end{multline*}
We thus arrived at the same conclusion as~\eqref{eq:yi close to y1}, and therefore~\eqref{eq:first term in cotype contra big} holds true. In combination with~\eqref{eq:contra assumption cotype R} and our current choice of $t$ in~\eqref{eq:second choice t}, we see that
$$
n\lesssim N_p(\R)^{\frac{2}{2-p}}n^{\frac{1-p}{2-p}},
$$
which simplifies to $n\lesssim N_p(\R)^2$, a contradiction for large enough $n$.
\end{proof}

\begin{corollary}\label{coro:no cotype banach}
Every Banach space $(X,\|\cdot\|_X)$ fails to have metric Markov cotype $p$ for all $p\in (0,2)$.
\end{corollary}

\begin{proof}
$X$ contains an isometric copy of $\R$. Since $\R$ is a $1$-absolute Lipschitz retract (see~\cite[Ch.~1]{BL}), it follows from Lemma~\ref{lem:retraction cotype} that $N_p(X)\ge N_p(\R)$. For $p\in (0,2)$ we have $N_p(\R)=\infty$ by Lemma~\ref{lem:no cotype R}.
\end{proof}

\subsection{Kalton's construction}

For $p\in [1,\infty)$ let $\ell_p^n$ denote (as usual) the space
$\R^n$ equipped with the  norm
$\|x\|_p=(|x_1|^p+\ldots+|x_n|^p)^{1/p}$. The unit ball of
$\ell_p^n$ is denoted below  $B_p^n=\{x\in \R^n:\ \|x\|_p\le 1\}$.
For every $\e\in (0,1)$ there exists a subset $\mathcal{N}$ of
$B_2^n$ such that $|\mathcal{N}|\le (1+2/\e)^n$ and $\min_{y\in
\mathcal{N}} \|x-y\|_2\le \e$ for every $x\in B_2^n$ (see
e.g.~\cite{MS}). In particular, there exists a linear operator
$Q_n:\ell_1^{5^n}\to \ell_2^n$ such that
\begin{equation}\label{eq:def Qn}
B_2^n\supseteq Q_n\left(B_1^{5^n}\right)\supseteq \frac12 B_2^n.
\end{equation}
Indeed, choose $\{x_1,\ldots,x_{5^n}\}\subseteq B_2^n$ with $\min_{i\in \{1,\ldots,5^n\}}\|x-x_i\|_2\le 1/2$ for every $x\in B_2^n$. This implies that the convex hull of $\{\pm x_1,\ldots\pm x_{5^n}\}$ contains $\frac12 B_2^n$. Hence, if we set $Q_n(e_i)=x_i$, where $e_1,\ldots,e_{5^n}$ is the standard basis of $\ell_1^{5^n}$, then the linear extension of $Q_n$ satisfies~\eqref{eq:def Qn}.

In what follows we fix a linear mapping $Q_n:\ell_1^{5^n}\to \ell_2^n$ for which~\eqref{eq:def Qn} holds true, and we also fix a mapping $\phi_n:B_2^n\to 2B_1^{5^n}$ such that $\phi_n(-x)=-\phi_n(x)$ for every $x\in B_2^n$ and  $Q_n\circ \phi_n$  is the identity mapping on $B_2^n$. The fact that such a $\phi_n$ exists is an immediate consequence of~\eqref{eq:def Qn}: simply choose a section $\varphi:B_2^n\to 2 B_1^{5^n}$ of $Q_n$ and define $\phi_n(x)=(\varphi(x)-\varphi(-x))/2$.

For every $\theta\in (0,1]$ consider the following linear subspace of $\ell_1^{5^n}\oplus \ell_2^n$.
\begin{equation}\label{eq:def Y_n}
Y_\theta^n\eqdef \left\{\left(\frac{x}{n^{\theta/4}},Q_n(x)\right):\ x\in \ell_1^{5^n}\right\}\subseteq \ell_1^{5^n}\oplus \ell_2^n.
\end{equation}
Below it will always be understood that $Y_n$ is equipped with the norm inherited from $(\ell_1^{5^n}\oplus \ell_2^n)_1$, i.e., $\|(x,y)\|_{(\ell_1^{5^n}\oplus \ell_2^n)_1}=\|x\|_1+\|y\|_2$ for every $(x,y)\in \ell_1^{5^n}\oplus \ell_2^n$.

Let $A_n\subseteq S^{n-1}=\partial B_2^n$ be a maximal (with respect to inclusion) symmetric set (i.e., $x\in A_n\iff -x\in A_n$)  such that $\|a-b\|_2>1/\sqrt[4]{n}$ for every distinct $a,b\in A_n$ . Thus $\min_{a\in A_n} \|x-a\|_2\le 1/\sqrt[4]{n}$ for every $x\in S^{n-1}$. Define a mapping $f_\theta^n:A_n\to Y_n$ by
\begin{equation}\label{eq:def fn}
\forall\, a\in A_n,\qquad f_\theta^n(a)\eqdef \left(\frac{\phi_n(a)}{n^{\theta/4}},a\right).
\end{equation}
Observe that $f_\theta^n(-a)=-f_\theta^n(a)$ for every $a\in A_n$.

\begin{lemma}\label{lem:f theta holder} For every $n\in \N$, $\theta\in (0,1]$ and $\tau\in [\theta,1]$ we have
$$
\forall\, a,b\in A_n,\qquad \|f_\theta^n(a)-f_\theta^n(b)\|_{Y_\theta^n}\lesssim n^{(\tau-\theta)/4} \|a-b\|_2^\tau.
$$
\end{lemma}
\begin{proof} We may assume that $a\neq b$, in which case $1/\sqrt[4]{n}\le \|a-b\|_2\le 2$. Consequently,
\begin{align*}
\|f_\theta^n(a)-f_\theta^n(b)\|_{Y_\theta^n}&=\frac{\|\phi_n(a)-\phi_n(b)\|_1}{n^{\theta/4}}+\|a-b\|_2\\&\le \frac{\|\phi_n(a)\|_1+\|\phi_n(b)\|_1}{n^{\theta/4}}+2^{1-\tau}\|a-b\|_2^\tau\\&
\le\left(4n^{(\tau-\theta)/4}+2^{1-\tau}\right)\|a-b\|_2^\tau.\qedhere
\end{align*}
\end{proof}

\begin{lemma}\label{lem:no holder ext}
Fix $n\in \N$, $L\in (0,\infty)$, $\theta\in (0,1]$ and $\tau\in [\theta,1]$. Suppose that $F:S^{n-1}\to Y_\theta^n$ satisfies \begin{equation}\label{eq:holder assumption F}
\forall\, x,y\in S^{n-1},\qquad \|F(x)-F(y)\|_{Y_\theta^n}\le Ln^{(\tau-\theta)/4} \|x-y\|_2^\tau,
 \end{equation}
 and that $F(a)=f_\theta^n(a)$ for every $a\in A_n$. Then $L\gtrsim n^{\theta/4}$.
\end{lemma}

\begin{proof} By replacing $F(x)$ with $(F(x)-F(-x))/2$ we may assume without loss of generality that $F(-x)=-F(x)$ for every $x\in S^{n-1}$. Since $F$ takes values in $Y_\theta^n$, it follows from~\eqref{eq:def Y_n} that there exists a mapping $\psi: S^{n-1}\to \ell_1^{5^n}$ such that
\begin{equation}\label{eq:introduce psi}
\forall\, x\in S^{n-1},\qquad F(x)=\left(\frac{\psi(x)}{n^{\theta/4}},Q_n(\psi(x))\right).
\end{equation}

Let $\sigma_{n-1}$ denote the normalized Haar measure on $S^{n-1}$. We claim that for every $y\in B_\infty^{5^n}=[-1,1]^{5^n}$ we have
\begin{equation}\label{eq:l1 norm bound psi}
\int_{S^{n-1}}|\langle \psi(x),y\rangle|d\sigma_{n-1}(x)\lesssim \frac{L}{n^{\tau/4}}.
\end{equation}
The proof of~\eqref{eq:l1 norm bound psi} is a standard application of the concentration of measure phenomenon on $S^{n-1}$. Indeed, consider the set $$
U_y\eqdef \{x\in S^{n-1}:\ \langle \psi(x),y\rangle \le 0\}\subseteq S^{n-1}.
$$
Since $\psi(-x)=-\psi(x)$ for every $x\in S^{n-1}$ we have $\sigma_{n-1}(U_y)\ge \frac12$. For $x\in S^{n-1}$ and  $t\in (0,\infty)$ note that
\begin{equation}\label{eq:far from Uy}
\langle \psi(x),y\rangle\ge t\implies \inf_{u\in U_y} \|x-u\|_2\ge \left(\frac{t}{Ln^{\tau/4}}\right)^{1/\tau}.
\end{equation}
Indeed, if $\langle \psi(x),y\rangle\ge t$ then for every $u\in U_y$, since $ \langle \psi(u),y\rangle \le 0$,
\begin{multline*}
t\le \langle \psi(x)-\psi(u),y\rangle \le \|\psi(x)-\psi(u)\|_1\\\stackrel{\eqref{eq:introduce psi}}{\le} n^{\theta/4}
 \|F(x)-F(u)\|_{Y_\theta^n}\stackrel{\eqref{eq:holder assumption F}}{\le} Ln^{\tau/4}\cdot\|x-u\|_2^\tau,
\end{multline*}
implying~\eqref{eq:far from Uy}. By the isoperimetric inequality on $S^{n-1}$ (see e.g.~\cite{MS}) it follows from~\eqref{eq:far from Uy} that
$$
\sigma_{n-1}\left(\left\{x\in S^{n-1}:\ \langle \psi(x),y\rangle\ge t\right\}\right)\lesssim \exp\left(-cn\left(\frac{t}{Ln^{\tau/4}}\right)^{2/\tau}\right),
 $$
where $c\in (0,\infty)$ is a universal constant. By symmetry, the same estimate holds true for $\sigma_{n-1}\left(\left\{x\in S^{n-1}:\ \langle \psi(x),y\rangle\le -t\right\}\right)$, and therefore
$$
\sigma_{n-1}\left(\left\{x\in S^{n-1}:\ |\langle \psi(x),y\rangle|\ge t\right\}\right)\lesssim \exp\left(-c\left(\frac{tn^{\tau/4}}{L}\right)^{2/\tau}\right).
 $$
Consequently,
\begin{multline*}
\int_{S^{n-1}}|\langle \psi(x),y\rangle|d\sigma_{n-1}(x)\lesssim \int_0^\infty \exp\left(-c\left(\frac{tn^{\tau/4}}{L}\right)^{2/\tau}\right) dt\\
= \frac{L\tau}{2c^{\tau/2}n^{\tau/4}}\int_0^\infty s^{\frac{\tau}{2}{-1}} e^{-s}ds=\frac{L\Gamma(1+\tau/2)}{c^{\tau/2}n^{\tau/4}}\lesssim \frac{L}{n^{\tau/4}},
\end{multline*}
completing the proof of~\eqref{eq:l1 norm bound psi}.

The Pietsch Domination Theorem~\cite{Pie67}  (see also~\cite[Prop.~3.1]{LP68}) implies that there exists a Borel probability measure $\mu$ on $B_\infty^{5^n}$ such that
\begin{equation}\label{eq:pietsch}
\forall\, x\in \ell_1^{5^n},\qquad \|Q_n(x)\|_2\le \pi_1(Q_n)\int_{B_\infty^{5^n}} |\langle x,y \rangle|d\mu(y),
\end{equation}
where $\pi_1(Q_n)$ is the $1$-summing norm of $Q_n$, i.e.,
$$
\pi_1(Q_n)=\sup_{k\in \N} \sup \left\{\sum_{i=1}^k \|Q_n(x_i)\|_2:\ \sup_{y\in B_\infty^{5^n}} \sum_{i=1}^k |\langle x_i,y\rangle|\le 1\right\}.
$$
A theorem of Grothendieck~\cite[Cor.~1]{Gro53} (see also~\cite[Thm.~4.1]{LP68}) implies that  $\pi_1(Q_n)\le K_G \|Q_n\|_{\ell_1^{5^n}\to \ell_2^n}$, where $K_G\in [1,2]$ is the Grothendieck constant. Recalling~\eqref{eq:def Qn}, we have $\|Q_n\|_{\ell_1^{5^n}\to \ell_2^n}\le 1$. Hence,
\begin{multline}\label{eq:fubini pietsch}
\int_{S^{n-1}} \|Q_n(\psi(x))\|_2d\sigma_{n-1}(x)\\\stackrel{\eqref{eq:pietsch}}{\lesssim} \int_{S^{n-1}}\int_{B_\infty^{5^n}}  |\langle \psi(x),y \rangle|d\mu(y)d\sigma_{n-1}(x)\stackrel{\eqref{eq:l1 norm bound psi}}{\lesssim} \frac{L}{n^{\tau/4}}.
\end{multline}
For $x\in S^{n-1}$ choose $a\in A_n$ such that $\|x-a\|_{2}\le 1/\sqrt[4]{n}$. Recalling~\eqref{eq:def fn} and~\eqref{eq:introduce psi}, since $F(a)=f_\theta^n(a)$ we have $Q_n(\psi(a))=a$. Hence,
\begin{multline}\label{eq:to contrast tau theta}
\|Q_n(\psi(x))\|_2\ge 1-\|Q_n(\psi(x))-Q_n(\psi(a))\|_2\\\stackrel{\eqref{eq:introduce psi}}{\ge} 1-\|F(x)-F(a)\|_{Y_\theta^n}
\stackrel{\eqref{eq:holder assumption F}}{\ge}  1-\frac{Ln^{(\tau-\theta)/4}}{n^{\tau/4}}=1-\frac{L}{n^{\theta/4}}.
\end{multline}
By combining~\eqref{eq:fubini pietsch} and~\eqref{eq:to contrast tau theta}, and recalling that $\tau\ge \theta$, the proof of Lemma~\ref{lem:no holder ext} is complete.
\end{proof}

\begin{proof}[Proof of Theorem~\ref{thm:no cotype}] The ensuing deduction of Theorem~\ref{thm:no cotype}
from Lemma~\ref{lem:no holder ext} follows an idea of~\cite{Nao01}. Fix $n\in \N$ and $p\in [2,n]$. Consider the metric on $\ell_2$ given by
$$\forall\, x,y\in \ell_2,\qquad d_p(x,y)\eqdef \|x-y\|_2^{2/p}.$$

Since by~\cite{Ball} the Markov type $2$ constant of $\ell_2$ satisfies $M_2(\ell_2)=1$, the Markov type $p$ constant of $(\ell_2,d_p)$ satisfies
\begin{equation}\label{eq:markov type of snowflake}
M_p(\ell_2,d_p)=M_2(\ell_2)=1.
\end{equation}
Lemma~\ref{lem:f theta holder} with $\theta=2/n$ and $\tau=2/p$ asserts that the function $$f_{2/n}^{2^{n^2}}:A_{2^{n^2}}\to Y_{2/n}^{2^{n^2}}$$ is $K$-Lipschitz in the metric $d_p$, where
$$
K\lesssim 2^{\frac{n^2}{4}\left(\frac{2}{p}-\frac{2}{n}\right)}.
$$
In light of~\eqref{eq:markov type of snowflake}, since the
$Y_{2/n}^{2^{n^2}}$ is finite dimensional it follows from
Corollary~\ref{coro:bona fide lip ext} that there exists a function
$$ F:\ell_2^{2^{n^2}}\to Y_{2/n}^{2^{n^2}}$$ that extends
$f_{2/n}^{2^{n^2}}$ and satisfies~\eqref{eq:holder assumption F}
with $\theta=2/n$, $\tau=2/p$ and
$$
L\lesssim N_p\left(Y_{2/n}^{2^{n^2}}\right).$$
We therefore deduce from Lemma~\ref{lem:no holder ext} that
\begin{equation}\label{eq:Np big fixed n}
\forall\, p\in [2,n],\qquad N_p\left(Y_{2/n}^{2^{n^2}}\right)\gtrsim 2^{n/2}.
\end{equation}
Consider the $\ell_1$ direct sum
$$
Y\eqdef \left(\bigoplus_{n=1}^\infty Y_{2/n}^{2^{n^2}}\right)_1.
$$
For every $n\in \N$ the restriction to the $n$th coordinate is a $1$-Lipschitz retraction from $Y$ onto $Y_{2/n}^{2^{n^2}}$, so by Lemma~\ref{lem:retraction cotype} we have
$$
\forall\, p\in [2,\infty),\qquad N_p(Y)\ge \sup_{n\in \N}  N_p\left(Y_{2/n}^{2^{n^2}}\right)\stackrel{\eqref{eq:Np big fixed n}}{=}\infty.
$$

By Corollary~\ref{coro:no cotype banach} we also have $N_p(Y)=\infty$ for $p\in (0,2)$, so in order to complete the proof of Theorem~\ref{thm:no cotype} it remains to show that $Y$ is isomorphic to a subspace of $\ell_1$. Since $Y$ is the $\ell_1$ direct sum of the spaces
$$
Y_{2/n}^{2^{n^2}}\subseteq \ell_1^{5^{2^{n^2}}}\oplus \ell_2^{2^{n^2}},
$$
it remains to recall that  $\ell_2^k$ is $(1+\e)$-isomorphic to a
subspace of $\ell_1$ for every $\e\in (0,1)$ and $k\in \N$ (e.g. by
Dvoretzky's theorem~\cite{Dvo60}).
\end{proof}

\begin{proof}[Proof of Theorem~\ref{thm:aspect}]
By Lemma~\ref{lem:f theta holder} and Lemma~\ref{lem:no holder ext} with $\theta=\tau=1$,
\begin{equation}\label{eq:for Z_n}
e\left(S^{n-1},A_n,Y_1^n\right)\gtrsim \sqrt[4]{n}.
\end{equation}
Since the diameter of $A_n$ is at most $2$ and the minimal nonzero distance in $A_n$ is at least $1/\sqrt[4]{n}$, the proof of Theorem~\ref{thm:aspect} is complete.
\end{proof}

Observe in passing that since $Y_1^n$ is $5^n$-dimensional, \eqref{eq:for Z_n} also implies~\eqref{eq:N2 Z_n lower} with $Z_{5^n}=Y_1^n$. For general $m\in \N$, choose $n\in \N$ such that $5^{n-1}\le m< 5^n$ and set $Z_m=\left(Y_1^{n-1}\oplus \ell_1^{m-5^{n-1}}\right)_1$.

\section{Comparison with Ball's approach}\label{sec:ball's cotype}

Fix $n\in \N$, $t\in [1,\infty)$ and let $A\in M_n(\R)$ be a
stochastic matrix that is reversible with respect to $\pi\in
\Delta^{n-1}$. Since $A$ has norm $1$ when viewed as an operator on
$L_2(\pi)$, we can consider the following matrix.
\begin{equation}\label{eq:defG}
\mathscr{B}_t(A)\eqdef \frac{1}{t}\sum_{s=1}^\infty \left(1-\frac{1}{t}\right)^sA^s\in M_n(\R).
\end{equation}
We also denote the corresponding Green's matrix by
$$
\G_t(A)\eqdef \frac{1}{t}I_n+\mathscr{B}_t(A)=\frac{1}{t}\left(I_n-\left(1-\frac{1}{t}\right)A\right)^{-1},
$$
where $I_n\in M_n(\R)$ denotes the identity matrix.

In~\cite{Ball} Ball worked with following {\em linear} invariant of
Banach spaces. For $p\in (0,\infty)$ say that a Banach space
$(X,\|\cdot\|_X)$ has Markov cotype $p$ with constant $N\in
(0,\infty)$ if for every $n\in \N$ and $t\in [1,\infty)$, if
$A=(a_{ij})\in M_n(\R)$ is a symmetric stochastic matrix that is
reversible relative to $\pi\in \Delta^{n-1}$ and $x_1,\ldots,x_n\in
X$ then
\begin{multline*}
(t-1)\sum_{i=1}^n\sum_{j=1}^n \pi_i a_{ij} \left\|\sum_{k=1}^n\greenmat_t(A)_{ik}x_k-
\sum_{\ell=1}^n\greenmat_t(A)_{j\ell}x_\ell\right\|_X^p\\
+\sum_{i=1}^n \pi_i \left\|x_i-\sum_{j=1}^n\greenmat_t(A)_{ij}x_j
\right\|_X^p \le N^p\sum_{i=1}^n\sum_{j=1}^n \pi_i
\mathscr{B}_t(A)_{ij} \|x_i-x_j\|_X^p.
\end{multline*}

As we shall see shortly, in Banach spaces Markov cotype $p$ implies  metric Markov
cotype $p$ as in Definition~\ref{def:markov cotype}, but their equivalence
remains open. Note that Ball proved in~\cite{Ball} that $\ell_1$ fails to
have Markov cotype $2$, so this question has relevance to
Question~\ref{Q:does l1 have cotype}.

In the closing remarks of his paper~\cite{Ball}, Ball proposed the
following two step definition of \emph{metric} Markov cotype for metric spaces.
First, given $p\in [1,\infty)$ say that a metric space $(X,d_X)$ is
$p$-approximately convex if there exists $K\in (0,\infty)$ with the
following property. Fix $m,n\in \N$ and let $B=(b_{ij})\in
M_{n\times m}(\R)$ and $C=(c_{ij})\in M_n(\R)$ be stochastic
matrices, such that $C$ is reversible relative to $\pi \in
\Delta^{n-1}$. Then for every $z_1,\ldots,z_m\in X$ there exist
$w_1,\ldots,w_n\in X$ such that
\begin{multline}\label{eq:def approximate convexity}
\sum_{i=1}^n\sum_{r=1}^m \pi_i b_{ir} d_X(w_i,z_r)^p+
\sum_{i=1}^n\sum_{j=1}^n \pi_i c_{ij} d_X(w_i,w_j)^p\\ \le K^p
\sum_{r=1}^m\sum_{s=1}^m (B^*D_\pi CB)_{rs} d_X(z_r,z_s)^p,
\end{multline}
where $D_\pi\in M_n(\R)$ is given as in~\eqref{eq:def diagonal},
i.e., it is the diagonal matrix whose diagonal equals $\pi$.
Assuming that $(X,d_X)$ is approximately convex, Ball defined it to
have metric Markov cotype $p$ if there exists $N\in (0,\infty)$ such that
for every $n,t\in \N$, if $A=(a_{ij})\in M_n(\R)$ is stochastic and
reversible relative to $\pi\in \Delta^{n-1}$ then for every
$x_1,\ldots,x_n\in X$ there exist $y_1,\ldots,y_n\in X$ such that
\begin{multline}\label{eq:def ball's cotype}
\sum_{i=1}^n\pi_i d_X(x_i,y_i)^p+(t-1)\sum_{i=1}^n\sum_{j=1}^n
a_{ij}
\pi_i d_X(y_i,y_j)^p\\
\le N^p\sum_{i=1}^n\sum_{j=1}^n \pi_i\mathscr{B}_t(A)_{ij}
d_X(x_i,x_j)^p.
\end{multline}
Denote by $N_p^B(X)$ the infimum over those $N\in (0,\infty)$ for
which~\eqref{eq:def ball's cotype} holds true.%

By Lemma~\ref{lem:CN} {every metric space} is a $p$-approximately
convex (with $K$ in~\eqref{eq:def approximate convexity} at most
$6$). So, the first step of Ball's definition is not needed. Observe
also that for Banach spaces Markov cotype $p$ trivially
implies~\eqref{eq:def ball's cotype}. Moreover, there is an
immediate link between~\eqref{eq:def ball's cotype}
and~\eqref{def:markov cotype}: due to~\eqref{eq:defG} we have
$\A_t(A)_{ij}\lesssim \mathscr{B}_t(A)_{ij}$ for every integer $t\ge
2$ and $i,j\in \{1,\ldots,n\}$. Therefore every metric space
$(X,d_X)$ satisfies $N_p^B(X)\lesssim N_p(X)$. Despite the fact that
one cannot bound from above (entry-wise) the matrix
$\mathscr{B}_t(A)$  by a constant multiple of the matrix $\A_t(A)$,
the following lemma implies that $N_p(X)\lesssim N_p^B(X)$.

\begin{lemma} Let $(X,d_X)$ be a metric space. Suppose that $n,t\in
\N$ and $A\in M_n(\R)$ is a stochastic matrix that is reversible
relative to $\pi\in \Delta^{n-1}$. Then for every $p\in [1,\infty)$
and $x_1,\ldots,x_n\in X$ we have
\begin{multline}\label{eq:desired p^p}
\sum_{i=1}^n\sum_{j=1}^n \pi_i \mathscr{B}_t(A)_{ij} d_X(x_i,x_j)^p\\\lesssim
p2^p \sum_{i=1}^n\sum_{j=1}^n  \pi_i \A_{\lceil pt\rceil}(A)_{ij}  d_X(x_i,x_j)^p.
\end{multline}
\end{lemma}

\begin{proof} Write
$
u\eqdef \lceil pt\rceil$ and note that for every $m\in \N$ we have
\begin{align}
&\nonumber
\sum_{i=1}^n\sum_{j=1}^n \pi_i (A^{mu})_{ij} d_X(x_i,x_j)^p\\&=\nonumber\sum_{i\in \{1,\ldots,n\}^{m+1}}
\pi_{i_1}\left(\prod_{a=1}^{m} (A^{u})_{i_a,i_{a+1}}\right)d_X(x_{i_1},x_{i_{m+1}})^p\\
&\le \label{eq:m holder}\sum_{i\in \{1,\ldots,n\}^{m+1}}
\pi_{i_1}\left(\prod_{a=1}^{m} (A^{u})_{i_a,i_{a+1}}\right)m^{p-1}\sum_{a=1}^{m} d_X(x_{i_a},x_{i_{a+1}})^p\\
&=\nonumber m^p \sum_{i=1}^n\sum_{j=1}^n \pi_i (A^{u})_{ij}d_X(x_i,x_j)^p\\
&\le \label{eq:use cesaro lemma} (2m)^p\sum_{i=1}^n\sum_{j=1}^n \pi_i\A_{u}(A)_{ij} d_X(x_i,x_j)^p,
\end{align}
where in~\eqref{eq:m holder} we used the triangle inequality and
H\"older's inequality, and in~\eqref{eq:use cesaro lemma} we used
Lemma~\ref{lem:control by cesaro}.

For every $m\in \N\cup \{0\}$ and $r\in \{0,\ldots,u-1\}$ we have
\begin{align}
&\nonumber\sum_{i=1}^n\sum_{j=1}^n \pi_i(A^{mu+r})_{ij}d_X(x_i,x_j)^p
\\&=\nonumber\sum_{i=1}^n\sum_{j=1}^n \sum_{k=1}^n \pi_i(A^{mu})_{ik}(A^r)_{kj} d_X(x_i,x_j)^p\\
&\le \nonumber 2^{p-1}\sum_{i=1}^n\sum_{j=1}^n
\sum_{k=1}^n \pi_i(A^{mu})_{ik}(A^r)_{kj} \big(d_X(x_i,x_k)^p+d_X(x_k,x_j)^p\big)\\
&=\nonumber 2^{p-1}\sum_{i=1}^n\sum_{j=1}^n\pi_i(A^{mu})_{ij} d_X(x_i,x_j)^p
+2^{p-1}\sum_{i=1}^n\sum_{j=1}^n\pi_i(A^r)_{ij} d_X(x_i,x_j)^p\\
&\le \frac12\sum_{i=1}^n\sum_{j=1}^n \pi_i\big((4m)^p\A_{u}(A)_{ij}+2^p(A^r)_{ij} \big) d_X(x_i,x_j)^p,
 \label{eq:use m path}
\end{align}
where in~\eqref{eq:use m path} we used~\eqref{eq:use cesaro lemma}.

Recalling the definition of $\mathscr{B}_t(A)$ in~\eqref{eq:defG},
by writing every $s\in \N$ as $s=mu+r$ for unique $m\in \N\cup
\{0\}$ and $r\in \{0,\ldots,u-1\}$, it follows from~\eqref{eq:use m
path} that
\begin{align}\label{eq:two term cesaro bound}
&\nonumber\sum_{i=1}^n\sum_{j=1}^n \pi_i\mathscr{B}_t(A)_{ij} d_X(x_i,x_j)^p\\\nonumber&=
\frac{1}{t}\sum_{i=1}^n\sum_{j=1}^n\sum_{m=0}^\infty\sum_{r=0}^{u-1}\left(1-\frac{1}{t}\right)^{mu+r}
\pi_i(A^{mu+r})_{ij}d_X(x_i,x_j)^p\\
&\nonumber\le \frac{4^p}{t}\left(\sum_{m=0}^\infty\sum_{r=0}^{u-1}
m^p\left(1-\frac{1}{t}\right)^{mu+r}\right)\sum_{i=1}^n\sum_{j=1}^n\pi_i\A_{u}(A)_{ij}d_X(x_i,x_j)^p\\
&\quad+ \frac{2^p}{t} \sum_{i=1}^n\sum_{j=1}^n \pi_i
\left(\sum_{m=0}^\infty\sum_{r=1}^{u}\left(1-\frac{1}{t}\right)^{mu+r}
(A^r)_{ij}\right)d_X(x_i,x_j)^p.
\end{align}

Recalling that $u=\lceil pt\rceil$, we have
\begin{align}
\nonumber \frac{1}{t}\sum_{m=0}^\infty\sum_{r=0}^{u-1}
m^p\left(1-\frac{1}{t}\right)^{mu+r}&\lesssim p\sum_{m=1}^\infty
\frac{m^p}{e^{pm}}\\&\le\label{eq:max of function} \frac{p}{e^p}+p\int_0^\infty \frac{x^p}{e^{px}}dx\\&=
\frac{p}{e^p}+\frac{\Gamma(p+1)}{p^p}
\lesssim \frac{p}{e^p} \label{eq:stirling},
\end{align}
where~\eqref{eq:max of function} uses the fact
that $x\mapsto x^pe^{-px}$ achieves its global maximum on
$[0,\infty)$ at $x=1$, and~\eqref{eq:stirling} uses Stirling's formula.

Next, for every $i,j\in \{1,\ldots,n\}$ we have
\begin{align}\label{eq:geometric sum}
\nonumber \frac{1}{t}\sum_{m=0}^\infty\sum_{r=1}^{u}\left(1-\frac{1}{t}\right)^{mu+r}
(A^r)_{ij}&\le \frac{u}{t} \left(\sum_{m=0}^\infty e^{-mu/t}\right)\A_u(A)_{ij}\\&\lesssim p \A_u(A)_{ij}.
\end{align}
\eqref{eq:desired p^p} now follows by
substituting~\eqref{eq:stirling} and~\eqref{eq:geometric sum}
into~\eqref{eq:two term cesaro bound}.
\end{proof}

\subsection*{Acknowledgements} We are grateful to the anonymous referees for their helpful suggestions. M.~M. was supported by ISF grants 221/07 and  93/11,
BSF grant 2010021, and NSF grant CCF-0832797. Part of this work was completed while M.~M. was a member of the Institute for Advanced Study at Princeton, NJ. A.~N.
was supported by NSF grant CCF-0832795, BSF grant 2010021, the Packard Foundation and the Simons Foundation. Part of this work was completed while A.~N. was visiting Universit\'e Pierre et Marie Curie, Paris, France.


\bibliographystyle{abbrv}
\bibliography{CAT0ext}

 \end{document}